\documentclass[11pt,reqno]{amsart}
\usepackage{amsmath,amsfonts,amssymb,leftidx,amsthm,bbding}
\usepackage[utf8]{inputenc}
\usepackage[T1]{fontenc}
\usepackage{hyperref}
\usepackage[all,cmtip]{xy}
\usepackage{amsmath}
\usepackage{tikz-cd}

\newtheorem{theorem}{Theorem}[section]
\newtheorem{cor}[theorem]{Corollary}
\newtheorem{lema}[theorem]{Lemma}
\newtheorem{prop}[theorem]{Proposition}
\newtheorem{example}[theorem]{Example}
\newtheorem{obs}[theorem]{Remark}
\newtheorem{defini}[theorem]{Definition}

\newtheorem{theof}{Theorem}

\newtheorem{conj}{Conjecture}

\newtheorem*{th*}{Question}

\newtheorem{thmx}{Theorem}

\theoremstyle{plain} 
\newcommand{\thistheoremname}{}
\newtheorem*{genericthm*}{\thistheoremname}
\newenvironment{namedthm*}[1]
  {\renewcommand{\thistheoremname}{#1}%
   \begin{genericthm*}}
  {\end{genericthm*}}

\numberwithin{equation}{section}

\newcommand{\C}{\mathbb{C}}
\newcommand{\N}{\mathbb{N}}

\newcommand{\R}{\mathbb{R}}
\newcommand{\Z}{\mathbb{Z}}
\newcommand{\A}{\mathbb{A}}   
\newcommand{\D}{\mathbb{D}}

\newcommand{\an}{\mathrm{an}}
\newcommand{\diam}{\mathrm{diam}}

\newcommand{\T}{\mathcal{T}}

\newcommand{\mor}{\mathrm{Mor}}
\newcommand{\San}{\mathsf{S}^\an}
\newcommand{\kob}{\mathrm{CK}}
\newcommand{\roy}{\mathrm{Roy}}
\newcommand{\dP}{d_{\mathbb{P}}}

\begin{document}

\title[Hyperbolicity notions] 
{Hyperbolicity notions for varieties defined over a non-Archimedean field}

\author{Rita Rodríguez Vázquez}
\address[Rita Rodríguez Vázquez]{CMLS, École polytechnique, CNRS, Université Paris-Saclay, 91128 Palaiseau Cedex, France
}
\email{rita.rodriguez-vazquez@polytechnique.edu}
\thanks{Research supported by the ERC grant Nonarcomp no. 307856.}

\date{December 26, 2017}

\begin{abstract}
 Firstly, we pursue the work of W. Cherry on the analogue of the Kobayashi semi distance $d_\kob$
 that he introduced for analytic spaces defined over a non-Archimedean metrized field $k$.
 We prove various characterizations of  smooth projective varieties for which  $d_\kob$ is an actual distance.

Secondly, we explore several  notions of  hyperbolicity for a  smooth algebraic curve $X$ defined
over $k$. 
We prove a non-Archimedean analogue of the equivalence between 
having negative Euler characteristic and 
the normality  of certain families of analytic maps taking values in $X$. 
\end{abstract}

\maketitle

\addtocontents{toc}{\protect\setcounter{tocdepth}{1}}
\tableofcontents

\section{Introduction}
The notion of Kobayashi hyperbolicity \cite{Kob67} is arguably one of the fundamental notions in complex geometry.
We refer to   \cite{Kobbook} and \cite{Langcomplex} for detailed monographics on the topic or to the more recent surveys \cite{VoisinLang,DiverioRousseau}.

It is a fundamental and remarkable fact that hyperbolic compact complex spaces can be characterized as follows:
 
\begin{theof}\label{thm equiv C}
Let $X$ be a smooth compact complex analytic space endowed with a hermitian metric.
The following conditions are equivalent:

\begin{enumerate}
\item The space $X$ is Kobayashi hyperbolic.
\item The derivative of any holomorphic map from the unit disk $\D$ to $X$ is bounded on every compact subset 
 $\mathrm{K}$ of $\D$ by a constant depending only on  $X$ and $\mathrm{K}$.
\item The space $X$ contains no entire curve.
\item The family $\mathrm{Hol}(\D, X)$ is normal. 
\end{enumerate}
\end{theof}

 For instance, a compact Riemann surface is hyperbolic if and only if its genus is at least two, and no abelian variety is hyperbolic. 

\smallskip

Recall that a family of holomorphic maps $f_n : \D \to X$ is normal if it is equicontinuous, 
and by Ascoli-Arzel\`a's theorem this means that 
up to extracting a subsequence, the sequence $f_n$ converges uniformly on every compact subset to a holomorphic map  $g: \D \to X$.
\smallskip
The main content of the above theorem is the implication (3)$\Rightarrow$ (1), known as  Brody's lemma  \cite{Brody}. 
 
 \medskip

Complex analytic spaces that are hyperbolic enjoy remarkable properties concering the compactness of the spaces of holomorphic maps with values in them.
De Franchis' theorem, generalized by 
 \cite{Samuel,KobOchiai,Noguchi},
 asserts that there exist only finitely many meromorphic surjective  maps from a compact variety into a compact hyperbolic variety.

\bigskip

We now fix a a non-Archimedean complete valued field $k$ that is nontrivially valued and algebraically closed. 
In this paper, we explore analogues of the previous theorem
for  analytic spaces defined over $k$. 
We shall work in the context of analytic spaces as developped by V. Berkovich in 
  \cite{Berk,Berk2}.
  Being  locally compact and locally pathwise connected, such spaces
  have good topological properties, what makes them an adapted framework to arguments of analytic nature.

\medskip

In \cite{Cherrykobayashi}, the author translates the definition  of Kobayashi chains on complex spaces \cite{Kob67} to the set of rigid points $X(k)$ of a Berkovich space $X$, which gives rise to the Cherry-Kobayashi semi distance $d_\kob$ on $X(k)$.
We will  say that an analytic space X is Cherry hyperbolic if $d_\kob$ is an actual distance. 
This semi distance shares several properties with its complex counterpart. 
On the unit disk, $d_\kob$ agrees with the standard distance, and $d_\kob$ is contracting for analytic maps.

In a series of papers
 \cite{Cherryphd,Cherrykobayashi,ACWalgdeg,Cherryabelian},
 Cherry studied in detail the behaviour of $d_\kob$ and the existence of entire curves in the case of abelian varieties and of projective curves, making extensive use of the reduction theory available for these varieties.
 His results contrast with the complex case. Indeed,  any abelian variety  $X$ is Cherry hyperbolic and contains no entire curve, i.e. every analytic map $\A^{1,\an} \to X$ is constant.

\medskip

In the non-Archimedean setting, limits of analytic maps need not be analytic so that the notion of normality has to be slightly modified, see  \cite{FKT}. Roughly speaking, a family of analytic maps from a 
 (boundaryless) analytic space $X$ to a  compact space  $Y$ is normal whenever 
 every sequence admits a subsequence  that is pointwise converging to a continuous map.
 In \cite{Montel} we proved  a version of Montel's theorem in this setting 
 for families of analytic maps with values in an affinoid space.
 
 \medskip

Inspired by a remarkable conjecture by Cherry formulated in \cite{Cherryabelian,Cherrykobayashi}
it is natural to ask whether the following holds.
\begin{conj}\label{conj cherry}
Let $X$ be a smooth compact boundaryless $k$-analytic space. 
The following conditions are equivalent:
\begin{enumerate}
\item The space $X$ is Cherry-Kobayashi hyperbolic.
\item The space $X$ contains no entire curve.
\item The space $X$ contains no rational curve.
\item The family  $\mor_k (\D, X)$ of analytic maps from $\D$ to $X$ is normal. 
\end{enumerate}
\end{conj}

 Cherry treated the equivalence between (1), (2) and (3) in the cases of curves, abelian varieties and certain large classes of compact algebraic surfaces in 
 \cite[ \S VII.3]{Cherryphd}.
 We continue the study of the relations between the previous properties under the assumption that the field $k$ has zero residue characteristic.
 This hypothesis is necessary 
to  control the size of the image of a disk under an analytic map in terms of the norm of its derivative.

\bigskip

We restrict our attention to smooth projective varieties, for which there is a natural way to measure the norm of a derivative.
Let $X$ be a smooth projective variety defined over $k$. We fix a projective embedding of $X\subset \mathbb{P}^N_k$, and consider the restriction of the spherical distance $d_\mathbb{P}$  to $X$, which allows us to define the Fubini-Study derivative $|f'(z)|$ of any analytic map $f: \D \to X$.

To simplify notations, we set $d^\prime_\kob := \min \{1, d_\kob \}$.
Our first result reads as follows:

\begin{thmx}\label{THM EQUIVALENCIAS}
Let $X$ be a smooth projective variety defined over an algebraically closed non-Archimedean complete field $k$ of residue characteristic zero. The following conditions are equivalent:
\begin{enumerate}\renewcommand{\labelenumi}{\roman{enumi})}

\item Every rigid  point $x \in X$  has a neighbourhood $U$ such that the semi distances $d^\prime_\kob$ and $d_{\mathbb{P}}$ are equivalent on $U(k)$.

\item The semi distance $d_\kob$ defines the same topology as  $d_{\mathbb{P}}$ on rigid points.

\item For every rigid point in the open unit disk $\D$ there exists a neighbourhood $U$ and a positive constant $C$ such that 
$$
\sup_{f \in \mathrm{Mor}_k(\D, X)} \sup_{z \in U} |f^\prime (z)| \leq C  ~.
$$
\end{enumerate}
\end{thmx}

The main content of this statement resides in the implication 
 ii) $\Rightarrow$ iii), which results from an adaptation of Zalcman's lemma  \cite{Zalcman} to the non-Archimedean setting.
After renormalizing, we obtain a sequence of analytic maps $g_n : \D(0;n) \to X$ whose Fubini-Study derivative is uniformly bounded on every compact subset of $\A^{1,\an}$ and does not vanish at   at $0$.
 A major difference with the complex case is that this sequence does not converge a priori to an entire curve  $\A^{1, \an} \to X$.

Notice that in the complex case, the assertion ii) is satisfied if and only if $X$ is Kobayashi hyperbolic by a theorem of Barth  \cite[I, \S 2]{Langcomplex}.

The hypothesis on the residue characteristic of $k$ is only used for the implication  iii) $\Rightarrow$ i).

\smallskip

Observe that condition iii) in Theorem \ref{THM EQUIVALENCIAS} implies that the family  $\mor_k(\D,X)$ 
of analytic maps from the open unit disk into $X$ is normal at every rigid point.
The converse implication seems likely to hold. 
Conjecture  \ref{conj cherry} shows that if the family
 $\mor_k(\D,X)$ is normal at every rigid point, then $X$ is  Cherry hyperbolic, 
 which is slightly weaker than condition ii).

Our next result answers affirmatively to Conjecture \ref{conj cherry} in the case of 
smooth projective curves over a field $k$ of zero residue characteristic.

\begin{thmx}\label{THM EQUIVALENCES COURBES}
Let $X$ be a smooth projective curve defined over an algebraically closed field of residue characteristic zero. The following conditions are equivalent:
\begin{enumerate}\renewcommand{\labelenumi}{\roman{enumi})}
\item The curve $X$ has positive genus.

\item   Every rigid  point $x \in X$  has a neighbourhood $U$ such that the semi distances $d^\prime_\kob$ and $d_{\mathbb{P}}$ are equivalent on $U(k)$.

\item The Fubini-Study derivative of $\mathrm{Mor}_k(\D, X)$ is uniformly bounded in a neighbourhood of every rigid point.

\item The family $\mor_k (\D, X)$ is normal.

\item The curve $X$ is Cherry hyperbolic.
\end{enumerate}
\end{thmx}

Observe that Theorem \ref{THM EQUIVALENCIAS} implies the equivalence between ii) and iii) and the implication ii) $\Rightarrow$ v).
For the other implications we rely on the semi-stable reduction theorem, 
which implies that the image of a  disk under an analytic map is again an analytic disk (except if $X$ is the projective line).

The equivalence between i) and v) without any hypothesis on the characteristic of $k$ was proved in \cite{Cherrykobayashi} 
by embedding a projective curve  into its Jacobian.

\bigskip

The previous theorems are not fully satisfactory, since abelian varieties satisfy all the equivalent conditions above, but carry many self-maps and thus  violate De Franchis' theorem and its generalizations.  It would be particularly interesting to characterize those varieties $X$ for which similar compactness properties hold for any family of analytic maps with values in $X$. Our next result addresses this problem in the case of smooth algebraic
curves.

Recall  any  smooth  (irreducible) algebraic curve $X$  can be uniquely  embedded
into a smooth projective curve $\bar{X}$ such that $\bar{X} \setminus X$ is a finite set of $k$-points. The Euler characteristic of $X$ is then defined by  
$$
\chi(X) = 2 -2g - \# (\bar{X} \setminus X)~,
$$
where $g$ denotes the genus of $\bar{X}$. 
Complex algebraic curves with negative Euler characteristic are precisely those that are Kobayashi hyperbolic, and for which
any family of holomorphic maps $Y \to X$ is normal  for every analytic space $Y$.

\smallskip

\begin{thmx}\label{THM HYPERBOLICITY CURVES}
Suppose that $k$ is a complete non-Archimedean  algebraically closed field of zero residue characteristic
whose  residue field  is countable. Let $X$ be a smooth irreducible algebraic curve over $k$. 

Then, the Euler characteristic $\chi(X)$ of $X$ is negative  if and only if  any sequence of analytic maps $f_n : U \to X$ admits a subsequence $f_{n_j}$ that converges pointwise to a continuous map $f_\infty: U \to \bar{X}$
such that  either $f_\infty(U) \subset X$ or $f_\infty$ is  constant  equal to a point in $\bar{X} \setminus X$. 
\end{thmx}

Recall that  $\chi(X)\le 0$ if and only if $\mor_k(\D,X)$ is normal, the projective case being a consequence of Theorem \ref{THM EQUIVALENCES COURBES}.

One implication was already  noticed in \cite{FKT}. 
When the Euler characteristic of $X$ is non-negative, then we may find a smooth boundaryless curve $U$ such that  the family of all analytic maps $\mor_k(U,X)$ is not normal.
When $X$ is the projective or the affine line, one can take $U$ to be the unit disk. When $X$ is the punctured affine line any open annulus works. 
In \S\ref{section elliptic}, we extend these arguments to any 
elliptic curve, in which case we may take $U= X$.  

\smallskip

The core of the proof lies in the forward implication: the family $\mor_k(U, X)$ is normal as soon as $\chi(X) <0$. 
Recall that the skeleton of a curve consists of the points that do not have a neighbourhood isomorphic to an open disk.
We consider first the case where the skeleton of $X$ is not too small and next the general case.

\begin{thmx}\label{THM NORMAL BAD REDUCTION}
Suppose that $k$ is a complete non-Archimedean  algebraically closed field of zero residue characteristic. Let $X$ be a smooth irreducible algebraic curve of negative Euler characteristic
whose skeleton $\San(X)$ is not a singleton.

Let $U$ be a smooth connected  boundaryless analytic curve.
Then there exists a finite affinoid cover $(\bar{X}_i)$ of $\bar{X}$ and a locally finite cover $(U_j)$ of $U$ by basic tubes such that for every analytic map $f: U \to X$ 
and every $j$ the image  $f(U_j)$ is contained in some affinoid $\bar{X}_i$.

Moreover, the affinoid cover $(\bar{X}_i)$ is independent of $U$.
\end{thmx}

Notice that this theorem shows a form of equicontinuity for maps from $U$ to $X$.
This result together with Montel's theorem \cite{Montel} implies a stronger form of Theorem~\ref{THM HYPERBOLICITY CURVES} when $\San(X)$ is not a singleton. 

\smallskip

When the skeleton of $X$ is reduced to a point, then $X$ is a projective curve and admits a smooth model over the valuation ring  $k^\circ$.
In other words, it is a curve with good reduction.
In this case, our arguments use in a crucial manner the hypothesis that the residue field $\tilde{k}$ is countable.

\bigskip

This paper is structured as follows. 
In \S \ref{section berko} we review some basic facts on Berkovich spaces and in \S \ref{section curves} on smooth analytic curves.
The Cherry-Kobayashi semi distance is introduced and discussed in \S  \ref{section cherry}.
Our non-Archi\-medean version of Zalcman's lemma is proved in \S \ref{section zalcman}.
Theorems \ref{THM EQUIVALENCIAS} and \ref{THM EQUIVALENCES COURBES} are proved in \S \ref{section cherry equiv}.
Section \ref{section maps curves} contains several preparatory results needed for 
Theorems   \ref{THM NORMAL BAD REDUCTION} and \ref{THM HYPERBOLICITY CURVES}, which are proved in \S \ref{section thm d} and \S \ref{section thm c} respectively.

\subsection*{Acknoledgements}
 I would like to thank Jérôme Poineau and William Cherry  for their valuable comments on an earlier version of this paper.

This research was supported by the ERC grant Nonarcomp no. 307856.

\section{Berkovich analytic spaces}\label{section berko}

In this section, we review some aspects of Berkovich analytic spaces. 
We fix once and for all a non-Archimedean non-trivially valued field $k$ that is algebraically closed.
We refer to \cite{Berk,Berk2,Temkinintro} for a thorough discussion of this theory.

\subsection{Good analytic spaces}

Given a positive integer $N$ and an $N$-tuple of positive real numbers $r=(r_1, \cdots, r_N)$,
we denote by $k\{r^{-1} T \}$ the set of power series $f=\sum_I a_I T^I$, $I=(i_1, \cdots, i_N)$, with coefficients $a_I\in k$ such that $|a_I| r^I \to 0$ as $|I|:= i_1+ \cdots + i_N$ tends to infinity.  The norm $\| \sum_I a_I T^I\| = \max_I |a_I|r^I$ makes $k\{r^{-1} T \}$ into a Banach $k$-algebra.
When $r=(1, \cdots, 1)$, this algebra is called the Tate algebra and we denote it by $\mathcal{T}_n$.

\smallskip

A Banach $k$-algebra $\mathcal{A}$ is called \emph{affinoid} if there exists an admissible surjective morphism of $k$-algebras $\varphi : k\{r^{-1} T \} \to \mathcal{A}$.
If  $r_i \in  |k^\times|$ for all $i$, then $\mathcal{A}$ is said to be strictly affinoid. 
It is a fundamental fact that all $k$-affinoid algebras are noetherian and that all their ideals are closed, see \cite[Proposition 2.1.3]{Berk}.
Notice that the fact that the epimorphism $\varphi$ is admissible implies that $\mathcal{A}$ and $k\{r^{-1} T \} /\ker (\varphi)$ endowed with the residue norm are isomorphic as Banach algebras.

\bigskip

The analytic spectrum $\mathcal{M(A)}$ of a Banach $k$-algebra $(\mathcal{A},\| . \|)$ is the set of all mutiplicative seminorms on $\mathcal{A}$ that are bounded by the norm $\| . \|$ on $\mathcal{A}$.
Given $f\in \mathcal{A}$, its image under a seminorm $x\in \mathcal{M(A)}$ is denoted by $|f(x)| \in \R_+$. 
The set $\mathcal{M(A)}$  is endowed with the weakest topology such that all the functions of the form $x\mapsto |f(x)|$ with $f\in \mathcal{A}$ are continuous. The resulting topological space  is  nonempty, compact and Hausdorff \cite[Theorem 1.2.1]{Berk}.

\medskip

Given a point $x \in \mathcal{M(A)}$, 
the fraction field of $\mathcal{A}/\mathrm{Ker}(x)$ naturally inherits from $x$ an absolute value extending the one on $k$.
Its completion is  the \emph{complete residue field at $x$} and denoted by $\mathcal{H}(x)$.

The analytic spectrum $X=\mathcal{M(A)}$ of a  $k$-affinoid algebra  $\mathcal{A}$ is called a $k$-affinoid space. 
When $\mathcal{A}$ is strictly affinoid, one says that $X$ is strictly affinoid.

 The affinoid space $X$ naturally carries a sheaf of analytic functions $\mathcal{O}_X$, see \cite[\S 2.3]{Berk}.

\begin{example} 
The closed polydisk  of dimension $N$ and polyradius $r = (r_1, \cdots, r_N)\in (\R^+_*)^N $ is defined to be $\bar{\D}^N (r) := \mathcal{M}(k\{r^{-1} T \})$.  
The Gauss point $x_g\in\bar{\D}^N$ is the point associated to the norm 
$$\left| ( \sum a_I T^I) (x_g) \right| :=  \max |a_I| . $$
When $r = (1, \cdots, 1)$ we just write $\bar{\D}^N$, and when $N=1$ we denote it by $\bar{\D}$. 
\end{example}

\begin{example}
Pick any real numbers $r\le R$. 
The closed annulus is the affinoid space $A[r, R] :=  \mathcal{M}( k\{R^{-1} T, r S  \} / (ST - 1))$.
It can be identified with the closed  subset of the closed disk $\bar{\D}(R)$ consisting of the points $x \in \bar{\D}(R)$ with $r \le |T(x) | \le R$.
\end{example}

In the following, we shall exclusively work with the category of \emph{good} analytic spaces which is formed by the subcategory of analytic spaces (defined in \cite{Berk})
that are locally ringed spaces modelled on affinoid spaces. 
In other words, any point in a good analytic space admits a neighbourhood isomorphic to an affinoid space.  

\begin{example}
The open polydisk of dimension $N$ and polyradius  $r \in (\R^+_*)^N $ is the set
$$\D^N_k(r) =\{ x \in \bar{\D}^N (r) : |T_i(x)| <r_i,  i=1, \ldots, N \}.$$ 
It can be naturally endowed with a structure 
of good analytic space by writing it as the  increasing union of $N$-dimensional polydisks $\bar{\D}^N_k(\rho)$ whose radii $\rho = (\rho_1, \cdots, \rho_N)\in (|k^\times|)^N $ satisfy $\rho_i < r_i$ for all $i=1, \ldots, N$.
\end{example}

\begin{example}
Pick any real numbers $r < R$. 
The open  annulus is the set 
$$A(r, R) =  \{ x \in A[r, R] : r < |T(x)| < R \}.$$
It can be naturally endowed with a structure 
of good analytic space by writing it as an  increasing union of closed annuli.
\end{example}

\subsection{Analytification of algebraic varieties}

A fundamental class of good analytic spaces are the analytifications of algebraic varieties.
To every algebraic variety $X$ over $k$ one can associate a $k$-analytic space $X^\an$ in a functorial way. We refer to \cite[\S 3.4]{Berk} for a detailed construction. 

\smallskip

In the case of an affine variety $X= \mathrm{Spec}(A)$, where $A$ is a finitely generated  $k$-algebra, then the set $X^{\an}$ consists of all the multiplicative seminorms on $A$ whose restriction to $k$ coincides with the norm on $k$. This set is endowed with the weakest topology such that all the maps of the form $x \in X^{\an} \mapsto |f(x)|$ with $f\in A$ are continuous. 
Fix an embedding of $X$ into some affine space $\A^{N}$. Then the intersections with the open polydisks $X^{\an} \cap \D^N(r)$ with $r>0$ define 
an analytic atlas on $X^\an$.

Observe that any $k$-point $x\in X$ corresponds to a morphism of $k$-algebras $A\to k$ and its composition with the norm on $k$
defines a rigid point in $X^{\an}$. Since $k$ is algebraically closed, one obtains in this way an identification of the set of closed points in $X$ with the set of rigid points in $X^{\an}$.

\medskip

Let $X$ be a general algebraic variety  and fix  an affine open cover.
The analytification of a general algebraic variety $X$  is obtained by glueing together the analytification of its affine charts in natural way. Analytifications of algebraic varieties are good analytic spaces, and closed points are in natural bijection with rigid points as in the affine case. 

We refer to the next section for a description of the topology of the analytification of an algebraic curve.

\subsection{Basic tubes}\label{section tubes}

Special analytic spaces will play an important role in the proof of Theorem \ref{THM EQUIVALENCES COURBES}. 
We thus introduce the following terminology.

\begin{defini}
A $k$-analytic space  $X$ is called a basic tube  if there exists a reduced equidimensional strictly $k$-affinoid space $\hat{X}$ and a  closed point $\tilde{x}$ in its reduction  such that $X$ is isomorphic to $\mathrm{red}^{-1}(\tilde{x})$.
\end{defini}

 By convention, a basic tube is reduced.

\begin{theorem}\label{thm:basic tube}
A basic tube is connected.
\end{theorem}

The fact that any basic tube over an algebraically closed field is connected is a deep theorem due to
\cite{Bosch}, which was generalized to arbitrary base fields in  \cite{Poineaucomposantes}.

\begin{example}
Let $a_1, \cdots, a_m$ be type II points in $\mathbb{P}^{1, \an}$.
Then every connected component of $\mathbb{P}^{1, \an} \setminus \{ a_1, \cdots, a_m\}$ is a basic tube.
\end{example}

Analytic spaces come with a natural notion of boundary and interior. We refer to  \cite[\S 3.1]{Berk} for the definitions for  good  $k$-analytic spaces and  to \cite[\S 1.5.4]{Berk2} for a discussion in the case of general Berkovich spaces.

Recall that a topological space is $\sigma$-compact if it is the union of countably many compact subspaces.  For instance,   open Berkovich polydisks or the analytification of an algebraic variety are $\sigma$-compact spaces. Observe that there exist simple examples of $k$-analytic spaces which are not
$\sigma$-compact, e.g. the closed unit disk of dimension $N
\ge 2$ with the Gauss point removed over a base field $k$ with uncountable reduction $\tilde{k}$.
The following propositions, whose proofs can be found in \cite{}, imply that every basic tube is  boundaryless and $\sigma$-compact.

\begin{prop}\label{prop tube cc aff}
A $k$-analytic space $X$ is a basic tube if and only if it is isomorphic to  a connected component of the interior of some equidimensional strictly $k$-affinoid space.
\end{prop}

\begin{prop}\label{basic tube intersection}
For every basic tube $X$ there exist a strictly $k$-affinoid space $\hat{X}$ and a distinguished closed immersion into some closed polydisk $\hat{X} \to \bar{\D}^N$ such that $X$ is isomorphic to $\hat{X} \cap \D^N$.
\end{prop}

\begin{obs}
Every good reduced boundaryless $k$-analytic space has a basis of open neighbourhooods that are basic tubes.
\end{obs}

\subsection{Normal families}
Recall the definition of normal family from \cite{FKT}.

\begin{defini}
Let $U$ be a boundaryless 
 $k$-analytic space and  $X$ a compact $k$-analytic space. 
A family of analytic maps  $\mathcal{F}$ from $U$ to $X$ is normal at a point  $z \in U$ if there exists a neighbourhood 
 $V \ni z$ where every sequence $\{ f_n\}$ in  $\mathcal{F}$ admits a subsequence $\{ f_{n_j}\}$  converging pointwise to a continuous map.
 The family 
 $\mathcal{F}$ is normal if it is normal at every point $z \in U$.
\end{defini}

We shall use the following result from \cite{Montel}:

\begin{theorem}\label{thm montel}
Let $k$ be a non-Archimedean complete   field that is non trivially valued and $X$ a good, reduced, $\sigma$-compact, boundaryless strictly $k$-analytic space. 
 Let  $Y$ be a strictly $k$-affinoid space. 

Then, every sequence of analytic maps $f_n: X \to Y$ admits a pointwise converging subsequence whose limit is continuous.
\end{theorem}

\section{Smooth analytic curves}\label{section curves}

In this section, we recall some facts on the structure of smooth analytic curves. 
Our main references are   \cite[\S 4]{Berk} and \cite{Ducroscourbes}.

\subsection{The analytic affine line}\label{section affine line}

Recall that  the analytic affine line $\A^{1,\an}$ is the  set of bounded seminorms on the polynomial ring $k[T]$.
The points in $\A^{1, \an}$  can be explicitly described as follows \cite[\S 1.4.4]{Berk}.

Pick $a\in k$ and $r \in \R_+$ and denote by 
$\bar{B}(a;r)$  the closed ball in $k$ centered at $a$ and of radius $r$.
To  $\bar{B}(a;r)$ one can associate a point $\eta_{a,r} \in \A^{1,\an}$
by setting $|P(\eta_{a,r})| := \sup_{|y-a| \le r} |P(y)|$ for every polynomial $P\in k[T]$. 
If $r = 0$, then $\eta_{a,0}$ corresponds to evaluating polynomials in $a \in k$.

More generally, any decreasing sequence of closed balls $\bar{B}(a_i; r_i)$ in $k$ defines a sequence of points $\eta_{a_i,r_i}$ that converges  in $\A^{1, \an}$
to a point $\eta \in \A^{1,\an}$ sending any polynomial $P \in k[T]$ to $|P(\eta)| = \lim_i |P(\eta_{a_i,r_i} ) |$.
Observe that such a sequence of balls might have empty intersection, in which case $\lim_i r_i = r >0$ since $k$ is complete.

It is a key fact due to Berkovich \cite[\S 1.4.4]{Berk} that  any point in $\A^{1, \an}$ comes from a decreasing sequence of closed balls in $k$.

\bigskip

Suppose that $x = \lim \eta_{a_i,r_i}$ and set $\bar{B} = \cap_i \bar{B}(a_i; r_i)$. V. Berkovich introduced the following terminology.
\begin{enumerate}\renewcommand{\labelenumi}{\roman{enumi})}
\item The point $x$ is of type I if and only if $\bar{B} = \{ a \}$, with $a \in k$.
\item The point $x$ is of type II if and only if  $\bar{B} = \bar{B}(a; r)$ with $r \in |k^\times|$.
 \item The point $x$ is ot type III if and only if $\bar{B} = \bar{B}(a; r)$ with  $r \notin |k^\times|$.
 \item The point $x$ is of type IV if and only if $\bar{B} = \emptyset$.
\end{enumerate}

Every point in $\A^{1,\an}$ falls into one of these four types.

\bigskip

The analytic projective line $\mathbb{P}^{1,\an}$ is the one-point compactification of $\A^{1,\an}$. An open (resp. closed) disk in $\mathbb{P}^{1,\an}$ is either an open (resp. closed)  disk in $\A^{1,\an}$ or the complement of a closed (resp. open) disk in $\A^{1,\an}$.
Connected affinoid domains in $\mathbb{P}^{1,\an}$ are the complement of finitely many open disks.

\subsection{First properties of analytic curves}

We extend the description of the previous section to arbitrary curves. 
Recall that a $k$-analytic curve $X$ is a $k$-analytic space that is Hausdorff and of pure dimension $1$.
Throughout this section, $X$ will denote a  smooth analytic curve over  $k$. 

\medskip

Points in a $k$-analytic curve can be classified as follows, see \cite[\S 3.3]{Ducroscourbes}.
Let $x \in X$ and 
  let $\mathcal{H}(x)$ be its complete residue field. Then one says that:
\begin{enumerate}\renewcommand{\labelenumi}{\roman{enumi})}
\item The point $x$ is of type I if  $\mathcal{H}(x) \simeq k$;
\item The point $x$ is of type II if  the reduction $\widetilde{\mathcal{H}(x)}$ has transcendence degree  $1$ over $\tilde{k}$ and $|\mathcal{H}(x)| = |k|$;
 \item The point $x$ is of type III if   $\widetilde{\mathcal{H}(x)} \simeq \tilde{k}$ and the value group $|\mathcal{H}(x)^\times|$ is generated by  $|k^\times|$ and some real number $r \notin |k^\times|$;
 \item The point $x$ is of type IV  if $\widetilde{\mathcal{H}(x)} \simeq \tilde{k}$, $|\mathcal{H}(x)| = |k|$ and $\mathcal{H}(x)$ is a non-trivial extension of $k$.
\end{enumerate}

On $\A^{1,\an}$, this classification of the points in an analytic curve  agrees with the one introduced above. 

\medskip

Let $x$ be a type II point in a $k$-analytic curve $X$. We define the genus $g(x)$ of the point $x$ as the genus of the unique smooth projective curve $C$ over $\tilde{k}$ whose field of rational functions is isomorphic to $\widetilde{\mathcal{H}(x)}$. 
The set of points in $X$ with positive genus is a closed discrete subset of $X$ by \cite[Th\'eor\`eme 4.4.17]{Ducroscourbes}.

\medskip

The following fundamental topological result will be used in the sequel, see \cite[Th\'eor\`eme 4.5.10]{Ducroscourbes}:

\begin{theorem}\label{paracompact}
Every $k$-analytic curve is paracompact.
\end{theorem}

\subsection{Graph structure}
Following the terminology of~\cite{Ducroscourbes}, we say that a locally compact Hausdorff topological space $X$ is a graph if it  admits a fundamental basis of open sets $U$ satisfying the following properties:

\begin{enumerate}\renewcommand{\labelenumi}{\roman{enumi})}
\item For every pair of points $x,y \in U$, there exists a unique closed subset $[x,y] \subset U$ homeomorphic to a segment of endpoints $x$ and $y$.

\item The boundary of $U$ in $X$ is finite.
\end{enumerate}

A graph $X$ is a tree if $X$ itself satisfies  property i).

It follows from the definition that graphs are locally path-connected and that trees are  path-connected.
And a topological space $X$ is a graph if and only if it is Hausdorff and every point has a neighbourhood that is a tree \cite[1.3.3.1]{Ducroscourbes}.

\smallskip

It is a fundamental fact that every $k$-analytic curve is a graph \cite[Th\'eor\`eme 3.5.1]{Ducroscourbes}.

\medskip

Let $X$ be a $k$-analytic curve. 
Given a point $x \in X$, the tangent space $T_x X$ at $x$  is defined as  the set of connected components of $U\setminus \{ x\}$
where $U$ is an open neighborhood of $x$ which is a tree. It can be also defined as 
the set  of paths leaving from $x$ modulo the relation having a common initial segment proving that the definition does not depend on the choice of $U$. 
Given any tangent direction $\vec{v} \in T_x X$, we denote by $U(\vec{v})$ the open subset of points $y \in X \setminus \{ x\}$ 
such that there exists a path starting from $y$ and abuting at $x$ in the direction of $\vec{v}$.

\subsection{Skeleton of an analytic curve}

Recall the definition of the skeleton of an analytic curve:

\begin{defini}
The skeleton of a curve $X$ is the set of all points $x\in X$ having no neighbourhood isomorphic to an open disk.
It will be denoted by $\San (X)$.
\end{defini}

The only smooth projective curve with empty skeleton is $ \mathbb{P}^{1, \an}$, see  \cite[\S 5.4.8]{Ducroscourbes}.
In the non-compact case,  examples of  curves with empty skeleton include the open disk and the affine line.
The skeleton of an open annulus $A(\rho, 1)$  is the segment consisting of the points $\eta_{0,r}$ for $\rho < r < 1$. 

\medskip

By definition, $\San (X)$ is a closed subset of $X$ and its complement is a disjoint union of open disks.
The skeleton of $X$ is a closed locally finite subgraph of $X$. 
If $X$ is projective, then $\San (X)$ is compact.

\medskip

For any curve $X$ with nonempty skeleton there is a retraction map $r_X : X \to \San (X)$, defined as follows.
Every point in the skeleton  is fixed by $r_X$.
 For every point $x \in X \setminus \San (X)$, denote by $U_x$ the maximal open neighbourhood of $x$  that is isomorphic to $\D$. 
Then, $r_X (x)$ is the unique point in the topological boundary of $U_x$ in $X$.
The retraction map is continuous.
\bigskip

Let us now introduce the notion of nodes of a curve $X$ following  \cite[Lemme 6.2.3]{Ducroscourbes},
which is a subset of its skeleton.
 In order to define it, recall that we say that a tangent direction $\vec{v} \in T_x X$ at a  type II or type III point $x \in X$ is  \emph{discal} if $U(\vec{v})$ is a disc.

\begin{defini}
Let $X$ be a smooth analytic curve over $k$.
A type II point $x \in \San (X)$ is a \emph{node} if 
 one of the following conditions is satisfied:
\begin{enumerate}\renewcommand{\labelenumi}{\roman{enumi})}
\item The point $x$ has positive genus;
\item There exist three distinct tangent directions  at $x$ that are  non-discal;
\item The point $x$ belongs to the boundary of $X$.
\end{enumerate}
\end{defini}

The set of all nodes of a curve $X$ will be denoted by $\mathsf{N}(X)$. 
Observe that it contains every branching point in $\San (X)$, which is discrete and closed as $\San(X)$ is a locally finite graph.
Since the boundary of $X$ is finite and the set of points $x \in X$ of positive genus is closed and discrete \cite[Th\'eor\`eme 4.4.7]{Ducroscourbes},  then  $\mathsf{N}(X)$ is also discrete and closed.
The complement of $\mathsf{N}(X)$ in $\San (X)$ is by definition a disjoint union of open segments. It follows that if $\mathsf{N}(X)$ is nonempty, then $X\setminus \mathsf{N}(X)$ 
is a disjoint union of infinitely many open disks, finitely many open annuli, and finitely many punctured disks.

\subsection{Smooth projective curves}\label{section projective curves}

Let us describe in more detail the structure of smooth projective curves. 
Recall that in this case, the skeleton is compact and the set of nodes finite.

Let $X$ be a smooth irreducible projective curve. 
Its genus $g$ is encoded in the topology of the skeleton and in the points of positive genus as follows, see  \cite[\S 5.2.6]{Ducroscourbes} and \cite[\S 4.3]{Berk}.
Denote by $b$ the first Betti number of $\San (X)$.
Then, one has the equality
\begin{equation}\label{eq genus}
g = b + \sum_{x \in X_{[2]}} g(x) ~,
\end{equation}
where $X_{[2]}$ is  the set of type II points in $X$.
Notice that the sum is finite, since a smooth projective curve $X$ has only finitely many points of positive genus.

\smallskip

We remark that the endpoints of the skeleton are nodes:

\begin{lema}\label{lema endpoint}
Let $X$ be a smooth projective curve.
Let $\eta$ be an endpoint of $\San (X)$. Then $\eta$ is a node and has positive genus.
\end{lema}

This result helps us describe all the possibilities for a smooth projective curve $X$ in terms of its skeleton and its nodes, see \cite[\S 5.4.12]{Ducroscourbes}.

\begin{enumerate}\renewcommand{\labelenumi}{\roman{enumi})}
\item If $X$ has empty skeleton, then it is isomorphic to $\mathbb{P}^{1,\an}$.

\item If $\San(X)$ is nonempty and $X$ has no nodes, then  
$\San (X)$ has no endpoints by Lemma \ref{lema endpoint}, and hence the skeleton must be homeomorphic to a circle. 
It follows from \eqref{eq genus} that $X$ has genus $1$. In that case, we say that $X$ is a Tate curve, i.e. the analytification of an elliptic curve whose $j$-invariant is not integral.

\item If $\mathsf{N}(X)$ is nonempty, then it follows from \eqref{eq genus} that $X$ has positive genus.
\end{enumerate} 

The particular case where  $X$ has only one node $\eta_X$ deserves a more thorough description.
 We distinguish two  possibilities for the geometry of $X$ based on its skeleton, which will be used in the proof of Theorem \ref{THM HYPERBOLICITY CURVES} and \ref{THM NORMAL BAD REDUCTION}.

\begin{enumerate}\renewcommand{\labelenumi}{\roman{enumi})}
\item The skeleton of $X$ consists only of the point $\eta_X$. In that case, we say that $X$ has good reduction.
By \eqref{eq genus}, the genus of $\eta_X$ equals the genus of the curve $X$. In particular, if $g(\eta_X) =1$ then
the curve $X$ is the analytification of an elliptic curve with bad reduction, i.e.  whose $j$-invariant is integral.

\item There is at least one loop in $\San( X)$ passing through $\eta_X$. By \eqref{eq genus}, $X$ has genus at least 2.
\end{enumerate}

\begin{proof}[Proof of Lemma \ref{lema endpoint}]
Let $\eta$ be an endpoint of the skeleton of $X$.
If $\San(X) = \{ \eta \}$, then $\eta$ is a node, since otherwise it has an open neighbourhood that is isomorphic to an open disk. As $X$ is boundaryless and $\San(X)$ has no branching points, we conclude that $\eta$ has positive genus.

We may assume that $\{ \eta \}$ is strictly contained in   $\San (X)$.
As $\eta \in \San (X)$, there exists a tangent direction $\vec{v} \in T_\eta X$ such that $U (\vec{v}) \cap \San (X)$ is non-empty, and so $\vec{v}$ is non-discal.
Being an endpoint of $\San (X)$, all the other tangent directions at $\eta$ are discal. 
Since $X$ is projective,  the only remaining possibility for $\eta$ to be a node is to have positive genus, since $\eta$ is not a branching point of the skeleton.

Suppose by contradiction that $g(\eta) = 0$. 
By \cite[Th\'eor\`eme 4.5.4]{Ducroscourbes}, then point  $\eta$ has a simply connected open neighbourhood $V$. 
Moreover, this theorem states that $V$ is isomorphic to an open  annulus if $\eta$ is a type III point.
If $\eta$ is a type II point, the same result implies that  we may reduce $V$ such that it is also isomorphic to an open  annulus, since $\eta$ is an endpoint of $\San(X)$.

In fact, as every tangent direction $\vec{v}^\prime \in  T_\eta X$ different from $\vec{v}$ is discal, we may assume that every $U(\vec{v}^\prime)$ is contained in $V$.
After maybe reducing $V$, we may assume that he topological boundary of $V$ in $X$ is  a single point in $\San(X)$  and  that no point in $V$ has positive genus.
By \cite[Proposition 5.1.18]{Ducroscourbes}, $V$ is isomorphic to an open disk, contradicting the fact that $\eta \in \San (X)$.
\end{proof}

We have the following description of curves of positive genus:

\begin{lema}\label{lema X minus N(X)}
Let $X$ be a smooth irreducible projective curve with non-empty set of nodes.
Then  $\San(X)$ can be decomposed as the disjoint union of $\mathsf{N}(X)$ and open segments $I_1, \ldots, I_a$, with each $I_j$ isomorphic to a real segment $(1, R_j)$ with  $ R_j\in |k^\times|$, for $1\le j \le a$.
The complement of $\mathsf{N} (X)$ in $X$ is 
a disjoint union of infinitely many open unit disks and annuli $A(1, R_1), \ldots, A(1, R_a)$.

Moreover, if $\mathsf{N}(X)$ consists of a single node $\eta_X$,  then the closure $\bar{I}_j= I_j \cup \{ \eta_X\}$ of each $I_j$ in $X$ is a circle.
\end{lema}

\begin{proof}
Let $X$ be a smooth projective curve with non-empty set of nodes.
Suppose that $\San (X) \setminus \mathsf{N}(X)$ contains a loop $C$. 
As the skeleton of $X$ is connected and every branching point in $\San (X)$ is a node, we conclude that $C = \San (X)$, contradicting the  fact that $\mathsf{N}(X)$ is nonempty.

Moreover, since $X$ is projective there are only finitely many nodes, which are all type II points.
As a consequence, the set $\San (X) \setminus \mathsf{N}(X)$ consists of finitely many open segments $I_j$ isomorphic to real segments $(1,R_j)$ with $R_j \in |k^\times|$, for $1\le j \le a$.
It follows that $X\setminus \mathsf{N}(X)$ consists of a disjoint union of open disks and the open annuli $A(1, R_1), \ldots, A(1, R_a)$.

Assume now that $X$ has only one node $\eta_X$. 
If $\San (X)$ consists only of the point $\eta_X$, then the complement of $N(X)$ is a disjoint union of open disks.
 Otherwise, the exists at least one loop in $\San (X)$ passing through $\eta_X$ by Lemma \ref{lema endpoint}.
 Necessarily, the closure of each segment $I_j$ in $X$ is $I_j \cup \{ \eta_X \}$, which is homeomorphic to a circle.
\end{proof}

 The following lemma will be essential in the subsequent chapters, specially the proof of Theorem \ref{THM HYPERBOLICITY CURVES}. It is a  particular case of \cite[Proposition 6.1.2]{Ducroscourbes}.

\begin{lema}\label{lema aff domain}
Let $X$ be a smooth irreducible projective curve over $k$ and $x \in X$ a rigid point.
Every open neigbourhood $U$ of $x$ has an open subset $V \subseteq U$ such that $X \setminus V$ is an affinoid domain of $X$.
\end{lema}

\subsection{Geometry of basic tubes of dimension $1$} \label{section basic open}

A one-dimensional basic tube has empty skeleton if and only if it is isomorphic to an open disk, see \cite[Proposition 5.1.18]{Ducroscourbes}.
An open annulus $A(R,1)$ with $R \in |k^\times|$ is a basic tube, and its skeleton is isomorphic to the open real segment $(R,1)$.

A point $x$ in the skeleton of a smooth analytic curve $X$ is a node if it satisfies one of the following conditions: the point $x$ is a branching point of the skeleton, $x$ has positive genus or it belongs to the boundary of $X$.

\medskip

An important class of basic tubes of dimension one are star-shaped domains:

\begin{defini}
A basic tube $U$ of dimension $1$ is called a  star-shaped domain if it is simply connected and contains exactly one node $\eta_U$.
\end{defini}

Let us describe the geometry of star-shaped domains in more detail.
Let $U$ be a star-shaped domain.
Since basic tubes are boundaryless, the point $\eta_U$ has positive genus or it is a branching point of $\San(U)$. 
The skeleton of $U$ can be decomposed in a disjoint union of $\{ \eta_U\}$ and finitely many open segments. 
Thus,  every connected component of $U \setminus \{ \eta_U\}$ is either isomorphic to an open disk or to an open annulus.
The latter correspond to the  non-discal tangent directions in  $ T_{\eta_U} U$.

As a consequence, a star-shaped domain  $U$ determines the following data:
\begin{enumerate}\renewcommand{\labelenumi}{\roman{enumi})}
\item The residue curve at $\eta_U$, which is the unique smooth projective curve $C_U$ over $\tilde{k}$ such that $\tilde{k}(C_U) \simeq \widetilde{\mathcal{H}(\eta_U)}$. The curve $C_U$ has genus  $g (U)$;

\item A reduced divisor $D_U$ on $C_U$ whose support is the set of non-discal directions at $\eta_U$;

\item For every non-discal direction $\vec{v} \in T_{\eta_U} U$
a  real number $\rho \in |k^\times|$ of norm less than $1$ such that the open set $U(\vec{v})$ is isomorphic to the open annulus $A(\rho, 1)$.
\end{enumerate} 

\subsection{Analytic maps between curves}

The following result will be systematically used in the sequel, see \cite[Lemme 6.2.4]{Ducroscourbes}:

\begin{lema}\label{lema discal}
Let $X$ be a smooth projective curve over $k$ and $U$ a basic tube of dimension $1$. 

Let $f: U \to X$ a non-constant analytic map.
Let $z \in U$ be a type II or III point and consider the tangent map $df (z): T_z U \to T_{f(z)} X$.
If a tangent direction $\vec{v} \in T_z U$ is discal, then so is $df(z)( \vec{v})$.
\end{lema}

As an immediate consequence, we have: 

\begin{lema}\label{lema image d'une boule est une boule}
Let $X$ be a smooth projective curve
and $f: \D \to X$  an analytic map.
Then the image of $f$ is contained in some connected component of $X\setminus \San (X)$.
\end{lema}

Another important fact that will be used throughout the present chapter is the following proposition: 

\begin{prop}\label{cor if you touch the skeleton you stay there}
Let $U$ be a  basic tube of dimension 1 that is  not analytically isomorphic to  the unit disk. 
Let $X$ be a smooth projective curve and $f: U \to X$  an analytic map.
If a point $z \in \San (U) \setminus \mathsf{N}(U)$ is such that $f(z)$  lies in the skeleton of $X$,
then the connected component of $\San (U) \setminus \mathsf{N}(U) $ containing $z$ is mapped to $ \San (X)$. 
\end{prop}

The proof relies on the following Lemma \cite[Lemme 6.2.5]{Ducroscourbes}:

\begin{lema}\label{lema preimage node}
Let $X$ be a smooth projective curve over $k$ and $U$ a basic tube of dimension 1. 

Let $f: U \to X$ a non-constant analytic map. 
If  $f(z)$ is a node in $\San (X)$, then $z$ is a node in  $\San (U)$.
\end{lema}

\begin{proof}[Proof of Proposition \ref{cor if you touch the skeleton you stay there}]
Let  $z \in \San (U)$ be a non-nodal point such that $f(z) \in \San (X)$. 
By Lemma \ref{lema preimage node}, the point $f(z)$ is not a node.
In particular, $f(z)$ cannot be an endpoint of the skeleton by Lemma \ref{lema endpoint}, and so both  $z$  and $f(z)$ have exactly two non-discal tangent directions. 

Consider the complement of $\mathsf{N}(U)$ in $\San(U)$, and let $I$ be the connected component containing the point $z$.
Suppose by contradiction that not the whole $I$  is mapped to the skeleton of $X$.
In this case, we may find a point $z^\prime \in I$ such that  $f(z^\prime) \in \San(X)$ and such that 
a non-discal direction $\vec{v} \in T_{z^\prime} U$ is mapped to some discal direction at $f(z^\prime)$.
The tangent map $df (z^\prime): T_{z^\prime} U \to T_{f(z^\prime)} X$ is surjective, and so there is a discal direction at $z^\prime$ that is mapped to a non-discal direction at $f(z^\prime)$. This  contradicts Lemma \ref{lema discal}.
\end{proof}

We shall use the following version of Hurwitz's theorem during the proof of Theorem \ref{THM HYPERBOLICITY CURVES}:

\begin{prop}\label{new isolated zeros}
Let $U$ be a  boundaryless connected curve over $k$ and $X$ a $k$-affinoid space.
Let $Z$ be any closed analytic subset of $X$.

Suppose that $f_n$ is a sequence of analytic maps from $U$ to $X\setminus Z$
converging pointwise to a continuous  
 map $g$. 

Then,  we have either that $g(U) \cap Z = \emptyset$ or $g(U) \subset Z$.
\end{prop}

\begin{proof}
The set $Z$ is the zero locus of some analytic function $\varphi$ in the affinoid algebra of $X$ with $|\varphi |_{\sup} \le 1$.
Since the zeros of $\varphi$ form a finite subset of $X$ of rigid points, we may assume that $X$ is the closed unit disk and that $Z$ is the origin by replacing $f_n$ by $\varphi \circ f_n$. 

In this case our assumption ensures that the  functions $h_n = \log |f_n| : X \to \R_-$ are harmonic. 
It follows from \cite[Proposition 3.1.2]{Thuillier} that either $h_n$ converges uniformly to $-\infty$ on compact subsets and so  $g(U) \subset \{0\}$,
or any limit map of the sequence $h_n$ is still harmonic, in which case one necessarily has $g(U) \cap \{ 0 \} = \emptyset$.
\end{proof}

\section{Cherry hyperbolicity}\label{section cherry}

\subsection{Cherry's notion of hyperbolicity}

Recall the definition of the Cherry-Kobayashi semi distance \cite{Cherrykobayashi}:

\begin{defini}
 Let $x,y $ be rigid points in a $k$-analytic space $X$.
  A Kobayashi chain joining $x$ and $y$ is a finite set of analytic maps  $f_l: \bar{\D} \to X$ and points $z_l, w_l \in \bar{\D}(k)$, $l=1, \cdots, m$ such that $  f_1(z_1)=x$, $ f_l(w_l) = f_{l+1} (z_{l+1})$ for $l=1, \cdots , m-1$ and $ f_m(w_m)= y$. The \emph{Cherry-Kobayashi semi distance on $X$} is defined by 
 $$d_\kob (x, y) = \inf \sum_{l=1}^m |w_l - z_l | ~,$$
 where the infimum is taken over all Kobayashi chains joining $x$ and $y$. If there is no Kobayashi chain joining $x$ and $y$, we set $d_	\kob(x,y) =\infty$. 
\end{defini}

Observe that the group of analytic automorphisms of $\bar{\D}$ is the set of series of the form $\sum_{n\ge 0} a_n T^n$ such that $|a_1| = 1$ and such that $\max_n |a_n| \le 1$, which are isometries for the distance $|.|$. Thus, $d_\kob$ is invariant under  automorphisms of $\bar{\D}$.
Up to composition by  such an automorphism
  we may suppose that $z_l=0$, for all $l$.

\smallskip

On the  closed disk $\bar{\D}$, the Cherry-Kobayashi semi distance 
agrees with the distance induced by the norm on $k$.

\begin{obs}
Not every point pair of points in an analytic space $X$ can be joined by a Kobayashi chain.
Assume for instance that $X$ has dimension at least two and take a point $x\in X$ such that the transcendance degree of $\widetilde{\mathcal{H}(x)}$ over $\tilde{k}$ greater than one. Then, no analytic map $\D \to X$ avoids the point $x$.
\end{obs}

\begin{defini}
A $k$-analytic space $X$ is \emph{Cherry hyperbolic} if $d_\kob$ is an actual distance on $X(k)$, which might take the value $\infty$.
\end{defini}

As $\A^{1,\an}$ can be written as a union of disks of whose radii tend to infinity, we see that $d_\kob$ is  exactly zero on $\A^{1, \an}$, just like over $\C$.

\smallskip

The semi distance $d_\kob$ shares an important property with its complex counterpart: analytic maps are distance decreasing with respect to  $d_\kob$.
As a consequence, the fact that $d_\kob$ is $0$ on the whole $\A^{1,\an}$ implies that Cherry hyperbolicity is stronger than Brody hyperbolicity, i.e. the non-existence of entire curves.

\medskip

Recall from \S \ref{section curves} that the skeleton $\San(X)$ of a curve $X$ is the set of points not having a neighbourhood that is isomorphic to an open disk.
Let $X$ be an elliptic curve and pick any two distinct  rigid points $x, y \in X$.
 If $x$ and $y$ belong to the same connected component of $X \setminus \San(X)$, then $d_\kob (x, y)$ agrees with the distance on the disk, and hence $d_\kob (x, y) \neq 0$.
Otherwise, they cannot be joined by a Kobayashi chain and thus $d_\kob (x, y) =\infty$. 
Thus, elliptic curves are Cherry hyperbolic.
The same argument shows that a  projective curve $X$ is Cherry-hyperbolic if and only if it  has strictly positive genus.

\medskip

We refer to \cite{Cherrykobayashi} for further details on the Cherry-Kobayashi semi distance.

\subsection{Alternate definition of the Kobayashi semi distance}

Let us briefly comment on the following alternate definition of an analogue of the Kobayashi semi distance on the set of rigid points of a $k$-analytic space $X$, which  is very natural. In the same notation as above, for any $x, y \in X(k)$ we set
 $$d (x, y) = \inf \max_{1\le l \le m} |w_l - z_l |~,$$
 where the infimum is taken over all Kobayashi chains joining $x$ and $y$. As before, if there is no Kobayashi chain joining $x$ and $y$ we set $d_\kob(x,y) =\infty$. 
 The obtained semi distance $d$ satisfies the ultrametric inequality.
 
 \begin{obs}
Let $X$ be a smooth analytic  curve, and pick any two rigid points $x, y$. 
Observe that there exists a Kobayashi chain joining $x$ and $y$ if and only if they belong to the same connected component of  $X \setminus  \San(X)$.
In this case, the chain consists of a single analytic map $f: \bar{\D} \to X$, and as a consequence we have that  $d = d_\kob$.
 \end{obs}

Clearly, $d \le d_\kob$ for any analytic space $X(k)$.
In general, these semidistances are not equivalent, as shown in the following example.

Indeed, pick a closed unit disk and take two rigid points $x, y \in \bar{\D}$ with $|x-y| = 1$.
Attach to $\bar{\D}$ two irreducible components $X_1$ and $Y$ isomorphic to  $\bar{\D}$, one passing through $x$ and the other through $y$.
For every integer $n \ge 3$, take rigid points $x_1^{(n)} \in X_1$, $y_n \in Y$ such that $|x-x_1^{(n)}| = |y-y_n | = \frac{1}{n}$. 
Attach a closed disk to $X_1$ passing through $x_1^{(n)}$.
Denote this new irreducible component by $X^{(n)}_2$ and pick a rigid point $x_2^{(n)} \in X^{(n)}_2$ with $| x_1^{(n)} - x_2^{(n)} | = \frac{1}{n}$.
Repeat this procedure as to obtain irreducible components $X^{(n)}_l \simeq \bar{\D}$  and rigid  points $x^{(n)}_l \in X^{(n)}_l$  for  $1 \le l \le n$
with $X^{(n)}_{n}=Y$, $x^{(n)}_{n-1} = y_n$ and $x^{(n)}_n = y$.
Denote by $x_0^{(n)} : = x$ for every $n \ge 3$.
Observe that $| x^{(n)}_l - x^{(n)}_{l-1}| = \frac{1}{n}$ for  every $l=1,\ldots, n$. 

For every $n$, the points $x = x_0^{(n)}, x_1^{(n)}, \ldots,  x^{(n)}_n = y$ form a Kobayashi chain joining $x$ and $y$. 
We see that 
$$d(x,y)  = \inf_{n \ge 3} |  x^{(n)}_1 - x^{(n)}_0 | =  \inf_{n \ge 3}  \frac{1}{n}= 0 ~, $$
 whereas
$$d_\kob (x,y) = \inf_{n \ge 3} \sum_{l=1}^n |  x^{(n)}_l - x^{(n)}_{l-1}| = \inf_{n \ge 3} n\cdot \frac{1}{n}= 1~.$$

\medskip

In the sequel, we shall build on W. Cherry's work and only consider the semi distance $d_\kob$. 

\subsection{Royden's length function}

Royden's length function on the tangent bundle of a complex manifold  has the particularity that the semi distance it defines on the manifold is precisely the Kobayashi semi distance. This enables us to translate the notion of Kobayashi hyperbolicity into infinitesimal terms. We refer to \cite{Royden} for further details. 
  
One can adapt this definition to Berkovich spaces as follows. Recall that the tangent space of a smooth analytic space $X$ at a point $x$ is the set of all derivations on the local ring $\mathcal{O}_{X,x}$. 

\begin{defini}
Let $X$ be a smooth analytic space over $k$. For every $x\in X(k)$,  Royden's length function is defined for every $ \vec{v} \in 
T_x X$ as 
\[
|\vec{v}|_\roy:= \inf \left\lbrace \frac{1}{|\lambda|} :  \exists f: \D \to X \mbox{ analytic}, f(0)=x, f^\prime(0) =\lambda \vec{v} \right\rbrace.
\]
\end{defini}

The following result holds for fields $k$ of arbitrary characteristic:

\begin{prop}
Let $X$ be a smooth projective variety over some non-Archimedean field $k$. 

If Royden's function is such that $|\vec{v}|_\roy = 0$ if and only if $\vec{v} = 0$, then $X$ contains no entire curve.
\end{prop}

\begin{proof}
Suppose there exists an entire curve $f: \A^{1, \an} \to X$.  We may suppose that $f$ is not constant at $0$, and hence $\vec{v}= f^\prime (0)\in T_{f(0)}X$ is a non-zero vector. Denoting by $m_n: \D \to \D(0;n)$ the homothety of ratio $n$,  the sequence $f_n := f\circ m_n: \D \to X$ is such that $f_n ^\prime (0) = \lambda_n \vec{v}$, with $|\lambda_n| = n$. Hence, $|\vec{v}|_\roy =0$.
\end{proof}

\begin{prop}
Let $X$ be a smooth projective variety defined over a complete non-Archimedean field $k$ of characteristic zero.

Assume that every rigid point admits a neighbourhood on which the Fubini-Study derivative of $\mathrm{Mor}_k(\D, X)$ is uniformly bounded. Then, Royden's function is such that $|\vec{v}|_\roy = 0$ if and only if $\vec{v} = 0$.
\end{prop}

\begin{proof}
Suppose that there exists a point $x \in X$ and some nonzero $\vec{v} \in T_x X$ such that $|\vec{v}|_\roy = 0$. Then there is a sequence $f_n :\D \to X$ fixing the origin and such that $f_n^\prime (0) = R_n v$, with $v \neq 0$ and $|R_n| \to +\infty$. Hence, the Fubini-Study derivative  explodes at 0.
\end{proof}

\section{Zalcman's reparametrization lemma}\label{section zalcman}

\subsection{Fubini-Study derivative}

We fix once and for all homogeneous coordinates on $\mathbb{P}^{N,\an}$ and $\mathbb{P}^{1,\an}$.

\begin{defini}
Let $\Omega$ be any open subset of  $\A^{1,\an}$.
Consider an analytic map $f: \Omega \to \mathbb{P}^{N, \an}$  and choose coordinates $f= [f_0: \cdots : f_N]$. 
The Fubini-Study derivative of $f$ at a point $z\in \Omega$, is
\begin{equation*}
| f^\prime (z) | = \max \{ 1, |z|^2\} \frac{\max |(f_i^\prime f_j - f_j^\prime f_i)(z)| }{\max |f_i(z)|^2} ~.
\end{equation*}
\end{defini}

Observe that, by construction, the function $z\in \Omega \mapsto | f^\prime (z) | $ is continuous.

Next we prove some properties of the Fubini-Study derivative. 

\begin{lema}\label{bound spherical derivative}
For every analytic map $f:\D \to \mathbb{P}^{N, \an}$ given in homogeneous coordinates by  $f= [f_0: \cdots : f_N]$, we have
\begin{equation*}
|f^\prime(z)| \leq \frac{\max\{ |f_i^\prime(z)|\} }{\max\{ |f_i(z)| \}} ~.
\end{equation*}
\end{lema}

The proof is trivial.

\medskip

Since rigid points are dense, for any analytic map $f: \Omega \to \mathbb{P}^{N, \an}$ we have: 
\begin{equation*}
\sup_{z\in \Omega} |f^\prime (z)| = \sup_{z\in \Omega(k)} |f^\prime (z)|~.
\end{equation*}
A direct computation shows:

\begin{lema}\label{composition}
Let $\Omega, \Omega^\prime$ be open subsets of $\A^{1,\an}$.
Consider analytic maps $g: \Omega \to \mathbb{P}^{1, \an}$, and $f:\Omega^\prime \to \mathbb{P}^{N, \an}$ with $g(\Omega) \subseteq \Omega ^\prime$.  Then, we have
\begin{equation*} 
|(f\circ g)^\prime(z)| =  |f^\prime(g(z))| \cdot |g^\prime (z)| ~.
\end{equation*}
\end{lema}

\begin{cor}\label{lema moebius sph}
For every $g\in \mathrm{PGL}(2, k^\circ)$ and every analytic map $f:\Omega \to \mathbb{P}^{1, \an}_k$, we have
$|(f\circ g)^\prime (z)| = |f^\prime (z)|$.
\end{cor}

\begin{proof}
We have to prove that $|g^\prime(z)| = 1$ for every $g\in \mathrm{PGL}(2, k^\circ)$. A simple calculation shows that it holds for every $g$ of the form $z\mapsto az$, $z\mapsto z+b$ and $z\mapsto 1/z$, with $|a|=1$ and $|b| \leq 1$. Since $\mathrm{PGL}(2, k^\circ)$ is generated by all the maps of this form, the assertion is proved.
\end{proof}

\subsection{Diameter function}

It will be useful in the sequel to estimate the size of the image of a disk. The required tool is the diameter function. 

There are diameter functions on $\A^{1, \an}$ and on $\mathbb{P}^{1, \an}$.
Recall from \S \ref{section affine line} that any point  $x \in\A^{1, \an}$ is uniquely determined by a decreasing sequence of disks $\{ \bar{B}(a_i; r_i) \}$ in $k$. 
The diameter function is defined as $\diam_{\A} (x) = \lim r_i$. On $\mathbb{P}^{1, \an}$, one sets  
\begin{equation*}
\diam (x) = \frac{\diam_{\A} (x)}{\max \{ 1, |x|^2 \} } ~.
\end{equation*}

 In both cases, a point has zero diameter if and only if it is rigid.

Observe that $\diam_{\A} (x) = \inf_{c\in k} |(T-c)(x)|$. We refer to  \cite[\S 2.7]{BR} for further details on the diameter function on $\mathbb{P}^{1,\an}$.

We now extend these definitions to  any dimension.

\begin{defini}
For  any $x \in \A^{N, \an}$,  we set:
\begin{equation*}
\diam_\A (x) :=\max_{1\leq i \leq N} \inf_{c_i\in k} |(T_i-c_i)(x)|  = \max_{1\leq i \leq N} \diam_{\A} \pi_i (x)~,
\end{equation*}
where $\pi_i : \A^{N, \an} \to \A^{1,\an}$ is the usual projection  to the $i$-th coordinate. 
\end{defini}

\begin{defini}
Let  $x \in \mathbb{P}^{N, \an}$. Choose an affine chart isomorphic to  $\A^{N, \an}$ at $x$.
We may assume $x= [1: x_1 : \cdots : x_N]$. Write $|x_i| = |T_i(x)|$.   We set:
\begin{equation*}
\diam (x) :=\frac{\diam_{\A} (x)}{\max \{ 1, |x_i|^2 \} }= \frac{\max_{1\leq i \leq N} \diam_{\A} \pi_i (x)}{\max \{ 1, |x_i|^2 \}}~.
\end{equation*}
\end{defini}

It is clear from the  definitions that $\diam_\A$ and $\diam$ are zero exactly on the rigid points. 
 
\begin{lema}\label{lema moebius diam}
The function $\diam: \mathbb{P}^{N, \an} \to \R_{\geq 0}$ is invariant under the action of $\mathrm{PGL}(N+1, k^\circ)$.
\end{lema}

\begin{proof}
The function $\diam$ is clearly invariant under translations of the form $T_i\mapsto T_i - a_i$, $a_i \in k^\circ$. 

It follows directly from the definition that for $a_1, \ldots, a_N \in k$ with $|a_i| = 1$,
\begin{equation*}
\diam(x_1, \cdots, x_N) = \diam (a_1x_1, \cdots, a_N x_N).
\end{equation*}

Finally, consider maps of the form 
\begin{equation*}
\varphi: [1: x_1: \cdots: x_N]\mapsto [x_i: x_1: \cdots: x_{i-1}: 1: x_{i+1}: \cdots: x_N].
\end{equation*}
Clearly, we have $\diam (\varphi(x)) = \diam (x)$.  
All these transformations generate $\mathrm{PGL}(N+1, k^\circ)$.
\end{proof}

\begin{lema}\label{lema diam}
Assume that $\mathrm{char}(\tilde{k})=0$, and consider an analytic map $f: \D\to \mathbb{P}^{N, \an}$.  Then for every $z\in \D$, we have
\begin{equation*}
\diam (f(z)) \leq \diam(z)\cdot |f^\prime(z)|~.
\end{equation*}
Moreover, for $N=1$ we have an equality.
\end{lema}

\begin{obs}
The previous lemma does not hold if $\mathrm{char}(\tilde{k})=p>0$. In fact,  there are maps with small Fubini-Study derivative and whose image is arbitrarily big. Take for instance the sequence $f_n: z\mapsto c_n z^{p^n}$, with $|c_n| = \left( p^n \right)^{p^n}$.
Denote by $\eta_{0,\epsilon}$ the point in $\D$ associated to the closed ball $\bar{B}(0;\epsilon)$. 
 A direct computation shows that 
\begin{equation*}  
  | f_n^\prime (\eta_{0,\epsilon}) |  = \frac{ |c_n| \epsilon^{p^n -1}}{p^n \max \{ 1, |c_n| \epsilon^{p^n} \}^2 } ~.
  \end{equation*}
 It follows that
\begin{equation*}
 \sup_{n}  \sup_{z\in \D} |f_n^\prime (z)| = 1 ~.
\end{equation*}
If $\epsilon <  p^{-n}$, then the Fubini-Study derivative is $| f_n^\prime (\eta_{0,\epsilon}) | = (p^n \epsilon)^{p^n - 1}$.
 Thus, we see that $\diam(f_n(\eta_{0,\epsilon}))$ cannot be bounded away from $1$ uniformly in $\epsilon$ and $n \in \N$.
\end{obs}

\begin{proof}[Proof of Lemma \ref{lema diam}]
Let us first consider the case $N=1$.  By continuity, it suffices to consider points of type II and III. 
More so, we may assume that  $z = \eta_{0,r} \in \D$ (i.e. $z$ is  associated to the closed ball $\bar{B}(0;r) \subset k$)
and that $f(z)=\eta_{0,R}$, 
since both the diameter function  and the Fubini-Study derivative are invariant under the action of $\mathrm{PGL}(2, k^\circ)$ by  Lemmas \ref{lema moebius sph} and \ref{lema moebius diam}.

Let $f(z) = \sum_{i\ge 0} a_i z^i$ be the series development of $f$.
 It follows from the definitions that the equality is equivalent to $\diam_\A (f(\eta_{0,r})) = \diam(\eta_{0,r}) \cdot |f^\prime (\eta_{0,r})|$. We have
\begin{equation*}
\diam_\A (f(\eta_{0,r})) = \max_{i\geq 1} |a_i| r^i = r \cdot \max_{i\geq 1} |a_i| \cdot |i| r^{i-1} = r \cdot |f^\prime (\eta_{0,r})|~,
\end{equation*}
concluding the proof for $N=1$.

Consider now the general case.
We may choose  homogeneous coordinates $[z_0, \cdots, z_N]$ in $\mathbb{P}^{N, \an}$ such that the inverse image  under $f$ of
 the hyperplane $H_\infty = \{ z_0=0 \}$ is a discrete subset of $\D$. 
Denote by $\Gamma$ its convex hull in $\D$. 
It suffices to prove the result for points $z \in \D$ lying outside $\Gamma$. 
On a neighbourhood $U$ of $z$ contained in $\D \setminus \Gamma$, the map $f$ can be expressed as a map $f: U \to \A^{N, \an}$, i.e. $f= [1: f_1, \cdots: f_N]$.

By the previous case, we know that 
\begin{eqnarray*}
\diam_\A (f(z)) 
&=& \max_{1\leq i \leq n} \diam(z)\cdot |( \pi_i \circ f)^\prime(z)|\\
&=& \diam(z) \max_{1\leq i \leq n} |f_i^\prime (z)|~.
\end{eqnarray*}
By Lemma \ref{bound spherical derivative}, we see that
\begin{equation*}
\diam (f(z)) = \diam(z) \frac{\max_{1\leq i \leq n} |f_i^\prime (z)|}{\max \{1, |f_i(z)|^2 \} } \leq \diam(z)\cdot |f^\prime(z)|~,
\end{equation*}
proving the assertion.
\end{proof}

\subsection{Zalcman's reparametrization lemma}

We follow the proof  found in \cite{Bert}.
Notice that our result does not imply that the reparametrized sequence is converging.

\begin{prop}\label{lema reparametrization}
Let $X$ be a smooth projective variety defined over an algebraically closed complete non-Archimedean 
  field $k$.

Suppose that there exists  a sequence of analytic maps $f_n: \D \to X$ whose Fubini-Study derivative is not locally uniformly bounded in a neighbourhood of some rigid point $z_0 \in \D$.
Then, we can find a sequence of rigid points $z_n \to z_0$ and a sequence $k \ni \rho_n \to 0$ such that the rescaled sequence $g_n(z) := f_n(z_n + \rho_n z)$ satisfies the following properties:
\begin{enumerate}\renewcommand{\labelenumi}{\roman{enumi})}
 \item Each $g_n$ is defined on  the open disk of radius $n$;
 \item The Fubini-Study derivatives of the maps $g_n$ are uniformly bounded on any compact subset of $\A^{1, \an}_k$;
 \item For every $n\in \N$ we have  $|g_n ^\prime (0) | = 1$. 
\end{enumerate}
\end{prop}

The proof relies on the following technical result, whose proof we transpose directly to the non-Archimedean setting.

\begin{lema}[Gromov]\label{lema gromov}
Let $\varphi: \bar{\D}(0; R) \to \R_+$ be a locally bounded function, and fix  $\epsilon >0$ and  $\tau >1$.
 Then, for every $a \in \bar{\D}(0; R)(k)$ such that $\varphi(a)>0$, there is  $b\in \bar{\D}(0; R)(k)$ satisfying:
\begin{enumerate}\renewcommand{\labelenumi}{\roman{enumi})}
\item $|a-b| \leq \frac{\tau} {\epsilon (\tau -1) \varphi(a)}$
\item $\varphi(b) \geq \varphi (a)$
\item If $x\in \bar{\D}(0; R)(k)$ is such that $|x-b| \leq \frac{ 1}{\epsilon \varphi(b)}$, then $\varphi(x) \leq \tau \varphi(b)$. 
\end{enumerate}
\end{lema}

\begin{proof}
Suppose we can find a point $a \in \bar{\D}(0; R)(k)$ such that every $b\in \bar{\D}(0; R)(k)$ fails to satisfy one of the three conditions.
In particular, so does $a$. As $a$ itself obviously satisfies i) and ii), there must exist a rigid point $a_1$ such that $|a_1-a| \leq \frac{ 1}{\epsilon \varphi(a)}$ and $\varphi(a_1) > \tau \varphi(a)$. 
 We will show by induction that  we can construct a Cauchy sequence of rigid points along which $\varphi$ is not bounded.
 
 Suppose that we have constructed $a_1, \cdots, a_n \in \bar{\D}(0; R)(k)$ satisfying $|a_i  - a| \leq \frac{ 1}{\epsilon \varphi(a)}$ and $\varphi(a_i) > \tau^i \varphi(a)$. In particular, $a_n$ satisfies i) and ii) and hence not iii). 
 We then find $a_{n+1} $ satisfying $|a_{n+1}-a_n| \leq  \frac{1}{\tau^n \epsilon \varphi(a)}$ and $\varphi(a_{n+1}) > \tau \varphi(a_n) > \tau^{n+1}\varphi(a)$. 
 The ultrametric inequality now shows that $|a_{n+1} - a| \leq \frac{1}{\epsilon \varphi(a)}$
 and that $|a_{n+j}-a_n| \leq  \frac{1}{\tau^n \epsilon \varphi(a)}$ for every positive integer $j$.
 Thus, $\{ a_n\}$ is a Cauchy sequence and  must  converge to some rigid point $\alpha$, but we have shown that $\varphi$ is not bounded at $\alpha$.
\end{proof}

\begin{proof}[Proof of Proposition \ref{lema reparametrization}]
We may suppose $z_0=0$, and $X= \mathbb{P}^{N, \an}_k$.

Pick a sequence  of rigid points $a_n \to 0$ such that $|f_n^\prime (a_n)|\geq n^3$.  For every $a_n$, we now apply Lemma \ref{lema gromov} chosing $\epsilon = 1/n$, $\tau_n = 1+\frac{1}{n}$ and $\varphi = |f_n ^\prime |$ and obtain a sequence $z_n \in \D(k)$ satisfying:

\begin{enumerate}\renewcommand{\labelenumi}{\roman{enumi})}
\item $|a_n-z_n| \leq \frac{n^2 + n} {|f_n^\prime (a_n)|} \leq \frac{2}{n}$;
\item $ |f_n^\prime (z_n)|\geq |f_n^\prime (a_n)| \geq n^3$;
\item If $x\in \bar{\D}(k)$ is such that $|x-z_n| \leq \frac{n}{|f_n^\prime (z_n)|}$, then
\begin{equation*}
 |f_n^\prime (x)| \leq (1+\frac{1}{n}) |f_n^\prime (z_n)|~. 
\end{equation*}
\end{enumerate}

It is clear that $z_n \to 0$. Now set $r_n =\frac{1}{|f_n^\prime (z_n)|}$, and pick $\rho_n \in k$ with $|\rho_n| = r_n$. We see that $r_n \leq \frac{1}{n^3}$, and hence $\rho_n \to 0$. Each map $g_n(z) := f_n(z_n + \rho_n z)$ is hence defined on $\D(0; n)$.
Fix some $R>0$ and pick $z \in \bar{\D}(R)$. 
We  compute using Lemma \ref{composition}: 
\begin{eqnarray*}
 |g_n^\prime(z)| & \le & R^2 \cdot r_n \cdot |f_n^\prime(z_n + \rho_n z)| \leq \\
 & \le &  R^2 \cdot r_n (1+\frac{1}{n}) |f_n^\prime (z_n)|  = R^2( 1+\frac{1}{n}) ~.
\end{eqnarray*}
The Fubini-Study derivative of the maps $g_n$ is thus  uniformly bounded on  compact sets.

Clearly,  $ |g_n^\prime(0)| = |\rho_n| \cdot | f_n^\prime (z_n) |= 1$ for all $n \in \N$.
\end{proof}

\section{Further notions of hyperbolicity}\label{section cherry equiv}

In an attempt to obtain hyperbolicity results analogous to complex ones, we may consider  other notions of hyperbolicity. 
As a first step, we compare the topologies on the set of rigid points induced by different semi distances.

\subsection{The projective distance}

Recall that for any two points given in homogeneous coordinates by $x=[x_0, \cdots , x_N]$, $y=[y_0, \cdots , y_N] \in \mathbb{P}^N_k$, one defines  their \emph{projective distance} as  
\begin{equation*}
\dP(x, y)= \frac{\max |x_i y_j - x_j y_i|}{\max|x_i| \max |y_j|}~.
\end{equation*}
By continuity, one can extend this definition to the whole Berkovich projective space $\mathbb{P}^{N, \an}$. This leads to the definition of discs in the projective space. 
We denote by $B_{d_\mathbb{P} } (x; R)$ the open polydisk for the projective distance centered at $x$ and of radius $R$.

Observe that $\dP (x,y) \le 1$ for every $x, y \in \mathbb{P}^{N,\an}$. Hence, if $R \ge 1$ one has that $B_{d_\mathbb{P} } (x; R) = \mathbb{P}^{N, \an}$.

\begin{lema}
Let  $f: \D \to \mathbb{P}^{N, \an}$ be an analytic map. Then, 
\begin{equation*}
\sup_{ x \in \D} |f^\prime (x)| \leq \sup_{x,y \in \D} \frac{\dP(f(x), f(y) ) }{|x-y|} ~.
\end{equation*}
\end{lema}

\begin{proof}
Fix some point rigid point $x	\in \D$. Using the Taylor series of $f$, we have that for $y$  close to $x$,
\begin{eqnarray*}
\left|
\begin{matrix}
  f_i(x) & f_i(y) \\
  f_j(x) & f_j(y)
 \end{matrix}
\right| & = &
 \left|
\begin{matrix}
  f_i(x) & f_i(x) + f_i^\prime (x) (x-y) + O((x-y)^2) \\
  f_j(x) & f_j(x) + f_j^\prime (x) (x-y) + O((x-y)^2)
 \end{matrix}
\right|=\\
&=& |x-y|
 \left|
\begin{matrix}
  f_i(x) &  f_i^\prime (x)  \\
  f_j(x) & f_j^\prime (x)  
 \end{matrix}
\right| + O((x-y)^2) ~.
\end{eqnarray*}
This means that 
\[
|f^\prime (x)| = \lim_{y\to x} \frac{\dP(f(x), f(y) ) }{|x-y|}~,
\]
and thus $|f^\prime (x)|  \le \sup_{y \in \D} \frac{\dP(f(x), f(y) ) }{|x-y|}$.
\end{proof}

\subsection{Topology induced by the projective distance}

Let $X$ be a  projective variety defined over a field $k$ of zero characteristic that is algebraically closed.
Fixing an embedding of $X$ into some projective space, we obtain a distance function on $X(k)$ induced by the pull-back of the Fubini-Study distance.
It is a fundamental fact that any two  embeddings $X \to \mathbb{P}^{N, \an}$ and $X \to \mathbb{P}^{M, \an}$ induce equivalent distances on $X(k)$, see e.g. \cite[Proposition 4.3]{Grieve}.

Our aim is to prove that on the set of rigid points of $X$ the topology
 induced by $\dP$  agrees with the  Berkovich topology.

\begin{prop}\label{cor top proj}
Let $X$ smooth projective variety  and fix an embedding  $X \to \mathbb{P}^{N, \an}$.

Then the Berkovich topology on the set of rigid points of $X$ agrees with the one induced by the projective  distance $d_{\mathbb{P}}$.
\end{prop}

To this end, we consider first the affinoid case.

\medskip

Recall that given two points $z=(z_1, \ldots, z_N), w = (w_1, \ldots , w_N)$ in $\bar{\D}^N (k)$, 
the usual  distance  is given by
$$d_{\D} (z, w) = \max_{1 \le i \le N} |z_i - w_i| ~ .$$

\begin{lema}
Let $X$ be a strictly $k$-affinoid space and fix a closed immersion $X \to \bar{\D}^N$.
The Berkovich topology on the set of rigid points of $X$ agrees with the one induced by the usual distance $d_{\D}$. 
\end{lema}

\begin{proof}
Pick a rigid point $x \in X$ and fix some positive number $\epsilon$. 
The open ball $B_{d_\D} (x; \epsilon)$ for the distance $d_{\D}$ centered at $x$ of radius $\epsilon$ can be expressed as the following finite intersection: 
\begin{equation*}
 B_{d_\D} (x; \epsilon) = \bigcap_{i= 1}^N \{ z\in X(k) : |(T_i - x_i ) (z) | < \epsilon \} ~.
\end{equation*}
For every $1\le i \le N$, the set $\{ z\in X(k) : |(T_i - x_i ) (z) | < \epsilon \}$ is an open set for the Berkovich topology.

\medskip

Conversely, pick any Berkovich open set $U$  in $X(k)$.
 We may assume that 
$U$ is a finite intersection of sets of the form $  \{ x \in X(k): r_i < |f_i(x)| < s_i \}$
for some analytic function $f_i \in \mathcal{O}(\bar{\D}^N)$ and some positive real numbers $r_i$ and $s_i$.
Recall that for any $z,w \in X(k)$ the following inequality holds:
\begin{equation*}\label{eq distance}
|f_i(z) - f_i(w)| \le \| f_i\| d_\D(z,w)~,
\end{equation*}
where $\| . \|$ denotes the norm on the Tate algebra $\mathcal{T}_N$.
As a consequence, 
$$U \cap X(k) =\bigcap_i \bigcup_{x \in U\cap X(k)} B_{d_\D} \left( x; ~ \frac{\min \{ \left| |f_i(x)| - r \right|, \left| |f_i(x)| - s \right| \} }{\| f_i \|} \right) ~,$$
and the result follows.
\end{proof}

\begin{proof}[Proof of Proposition \ref{cor top proj}]
Pick a rigid point $x \in X$ and fix some positive real number $\epsilon <1$. 
The open ball $B_{d_\mathbb{P}} (x; \epsilon)$  for the projective distance can be expressed as a finite intersection of open sets for the Berkovich topology as follows: 
\begin{equation*}
 B_{d_\mathbb{P}} (x; \epsilon) = \bigcap_{\substack{0 \le i,j \le N \\ i \neq j } } \{ y\in X(k) : |(x_i T_j - x_j T_i ) (y) | < \epsilon \} ~.
\end{equation*}
The converse follows from the fact that the projective space can be covered by a finite number of Berkovich polydisks.
\end{proof}

Given a projective variety $X$, we may consider the semi distance $d^\prime_\kob := \min \{1, d_\kob \}$.
We now compare it with the projective distance.

\begin{prop}\label{prop proj cherry}
Let $X$ be a smooth projective variety. 
For any rigid point $x \in X$, there exists an open neighbourhood $U$ of $x$ and a positive constant $C$ such that $d^\prime_\kob \le C \dP$ on $U(k)$.
\end{prop}

\begin{proof}
Denote by $M$ the dimension of $X$.
Pick a rigid point $x\in X$ and fix an analytic map $\varphi: \D^M \to X$  that is an isomorphism on its image, sending some $z \in \D^M$ to $x$.
Set $U:= \varphi (\D^M)$.
Embed $X$ in some projective space $\mathbb{P}^{N,\an}$. 
By Proposition \ref{cor top proj}, we may choose a positive number $\epsilon$ such that 
$B_{d_\mathbb{P}}(x; \epsilon ) \cap X(k)$ is contained in $U$.

After maybe reducing the polydisk $\D^M$, we may assume that  $U$ is contained in some fixed unit polydisk $\bar{\D}^N \subset \mathbb{P}^{N,\an}$. 
Thus, the projective distance agrees with the usual distance on $U$.
Notice that the map $\varphi$ is given by $\varphi= ( \varphi_1  , \ldots , \varphi_N )$, where every  $\varphi_i \in \T_M$ has coefficients bounded by $1$, for $1\le i \le N$. 
Given any rigid point $y \in \bar{\D}^N$ with $y = \varphi (w) $ for some $w \in \D^M$, we have that 
\begin{equation*}
\dP (x,y) = \dP (\varphi (z), \varphi (w) ) = \max_{1\le i \le N} \left| \varphi_i (z) - \varphi_i (w) \right| ~.
\end{equation*}

For distinct $z, w \in \D^M (k)$, 
consider the real-valued function 
$$\Theta(z,w) = \frac{\dP (\varphi(z), \varphi (w) )}{d^\prime_\kob (z, w )} ~.$$
This function is strictly positive.
By the previous equation, we know that $\Theta(z,w) = \frac{\max_{i\le N} \left| \varphi_i (z) - \varphi_i (w) \right|}{\max_{j\le M} |z_j - w_j| }$.
The Taylor series development of each component $\varphi_i$ implies that for any $z, w \in \D^M (k)$, we may write $\varphi_i (w) - \varphi_i (z) = \sum_{j= 1}^M (z_j - w_j) \partial_j \varphi_i (z) + O (\sum_{1 \le j \le M} |z_j - w_j|^2)$, 
where $\partial_j \varphi_i (z)$ denotes the partial derivative of $\varphi_i$ with respect to the $j$-th component.
Using this observation, we may extend the function $\Theta$ continuously to the diagonal by setting
 $$ \Theta (z, z) = \lim_{w\to z} \Theta(z,w) = \max_{\substack{1 \le i \le N \\ 1 \le j \le M }} | \partial_j \varphi_i (z) | ~.$$
As $\varphi$ is an isomorphism on its image, not all the partial derivatives  $\partial_j \varphi_i$ are zero at the same time, and so $\Theta$ is strictly positive on the whole $\D^M (k) \times \D^M (k)$.
We may so find a positive constant $C$ such that $\Theta(z,w ) \ge C$ for every $z,w \in \D^M(k)$.
As the Cherry-Kobayashi semi distance contracts analytic maps, we see that $d_\kob (\varphi (z), \varphi (w) ) \le d_\kob (z,w)$. 
Thus, 
$d^\prime_\kob  (\varphi (z), \varphi (w) ) \le C \dP  (\varphi (z), \varphi (w) )$.
\end{proof}

\subsection{Proof of Theorem \ref{THM EQUIVALENCIAS}}

We consider the following notion of hyperbolicity that arises naturally from  the Cherry-Kobayashi semi distance:

\begin{defini}
Let $X$ be a smooth projective variety defined over an algebraically closed non-Archimedean complete field.
 The variety $X$ is strongly Cherry hyperbolic if the semi distance $d_\kob$ defines the same
topology as the projective distance on rigid points.
\end{defini}

We shall see in Theorem \ref{THM EQUIVALENCIAS} that  this notion is  stronger than that of Cherry hyperbolicity. 
If $X$ is a hyperbolic complex analytic space, Barth showed that the Kobayashi metric defines the topology of $X$, see \cite[Theorem \S I.2.3]{Langcomplex}.

\begin{proof}[Proof of Theorem \ref{THM EQUIVALENCIAS}]
i) $\Rightarrow$ ii): 
 Let us first show that $X$ is Cherry hyperbolic. 
Pick any point $x\in X(k)$. We shall prove that $d_\kob(x,y) >0$ for all $y\in X(k)$ different from $x$.

By Corollary~\ref{cor top proj}, the topology induced by $d_{\mathbb{P}}$ agrees with the Berkovich topology. Our assumption i) thus implies the existence of
 $\epsilon >0$ and a constant $C>0$ such that $d'_\kob (x_1, x_2) \ge C d_{\mathbb{P}}(x_1, x_2)$ whenever \[\max \{ d_{\mathbb{P}} (x,x_1), 
d_{\mathbb{P}} (x,x_2)\} \le \epsilon.\]

If $y$ is such that  $d_{\mathbb{P}} (x,y) \le \epsilon$, then $d'_\kob (x, y) \ge C d_{\mathbb{P}}(x, y)>0$ as required.
Suppose now that $d_{\mathbb{P}} (x,y) >\epsilon$, and pick any Kobayashi chain joining $x$ and $y$. We get a finite set of analytic maps  $f_l: \bar{\D} \to X$ and points $z_l \in \bar{\D}(k)$, $l=1, \cdots, m$ such that $  f_1(0)=x$, $ f_l(z_l) = f_{l+1} (0)$ for $l=1, \cdots , m-1$ and $ f_m(z_m)= y$. We shall prove that $|z_l| \ge C\epsilon/4$ for some $l$, which proves that $d_\kob (x,y) \ge C\epsilon/4$.

For each $l$, consider the function $d_l(t) := d_{\mathbb{P}}(x, f_l(t))$. This is a continuous function on the whole disk $\bar{\D}$. Since all closed disks $\bar{\D}(0;  |z_l|)$ are connected, the subset of the real line
$\cup_{l=1}^m d_l ( \bar{\D}(0;  |z_l|))$ is also connected.
As $d_1(0) = 0$ and $d_m(z_m) > \epsilon$,  we may find an integer $l$ and a point $\tau \in \bar{\D} (0;  |z_l|)$ such that 
$d_l(\tau)  = \epsilon/2$. By density of rigid points in the open disk, we may find $t \in  \bar{\D}(0; |z_l|)(k)$ such that $d_l(t) \in (\epsilon/4, 3\epsilon/4)$. 
We get that $|z_l| \ge |t| \ge d_\kob (x, f_l(t)) \ge C \epsilon/4$. 

We now prove that the two topologies induced by $d'_\kob$ and $d_{\mathbb{P}}$ are the same. This amounts to checking that converging sequences for one topology 
are converging for the other one. Suppose first that $d_{\mathbb{P}}(x_n, x) \to 0$. Then for sufficiently large $n$ we have that $x_n$ lies in a neighbourhood of $x$  where
$d'_\kob$ is equivalent to $d_{\mathbb{P}}$, hence $d'_\kob(x_n, x) \to 0$. 

Suppose next that $d'_\kob(x_n, x) \to 0$. Our arguments above show that for  $n$ sufficiently large $x_n$ belongs to a neighbourhood of $x$ on which 
$d'_\kob$ is equivalent to $d_{\mathbb{P}}$, so that  again $d_{\mathbb{P}}(x_n,x)\to 0$.

\medskip

ii) $\Rightarrow$ iii): 
Fix an embedding of $X$ in some analytic projective space $\mathbb{P}^{N, \an}$.
Suppose that the Fubini-Study derivative explodes at some point of $\D$.
We apply Proposition \ref{lema reparametrization} to  find a sequence of analytic maps $g_n: \D(0;n) \to X$ satisfying $|g_n^\prime (0)| = 1$ and with  uniformly bounded Fubini-Study derivative on compact subsets of $\A^{1,\an}$. 

Denote by $d_n = \diam (g_n(\D))$.
 If  $d_n$ tends  to zero as $n$ goes to infinity, then after maybe extracting a subsequence  all the $g_n(\D)$ are contained in some fixed ball of $\mathbb{P}^{N, \an}$. 
Schwarz' lemma implies that the derivative at zero is strictly smaller than 1, contradicting the fact that $|g_n^\prime (0)| = 1$. Thus, we may assume that there exists some $\epsilon >0$ such that  $d_n > \epsilon$ for every $n\in \N$.
 In particular, for every $n$ there are rigid points $w_n, z_m \in \D$ such that $\dP (g_n(w_n), g_n(z_n)) \geq \frac{\epsilon}{2}$. However,  $d_\kob (g_n(w_n), g_n(z_n)) \leq d_\kob (w_n, z_n) \leq \frac{1}{n}$, since $g_n$ is defined on $\D(0;n)$, and so the distances $d_\kob$ and $\dP$ cannot be equivalent.

\medskip

iii) $\Rightarrow$ ii):
Suppose that  the Fubini-Study derivative of all the analytic maps  $f: \D \to X$ is uniformly bounded on some open disk $\D(0; r)$ by some positive constant $C$.
It suffices to show that given any rigid points  $x_n, x$ in $X$  such that $d_\kob (x_n, x)$ tends to  $0$, $\dP (x_n, x) \to 0$ as $n$ goes to infinity.
 
  For every $n$,  consider a Kobayashi chain $f_l^n: \D \to X$, $l=1, \cdots, N_n$, joining $x_n$ and $x$ of length $r_n <r$ and such that $r_n \to 0$.  By Lemma \ref{lema diam}, we see that $f_l^n(\D(0; r_n) ) \subseteq \D^N(f_l^n(0); r_n C)$.
For every fixed $n \in \N^*$,   these polydisks have nonempty intersection by definition of Kobayashi chain and have the same radius, and so they  must be the same. Hence, $\dP (x_n, x)$ tends to  $0$.

\medskip

iii) $\Rightarrow$ i): 
Let  $r>0$ and $C >0$ be constants such that
\begin{equation*}
\sup_{\mathrm{Mor}_k(\D, X)} \sup_{\D(0; r)} |f^\prime (z)| \leq C  ~.
\end{equation*}
After maybe reducing the radius $r$, we may assume that $d^\prime_\kob \le \dP$ on $\D(0; r)$ by Proposition \ref{prop proj cherry}.
Pick any two rigid points $x, y \in X$.
Assume that $d_\kob (x, y)  \le \frac{r}{2}$. 
Let  $C_n$ be a sequence of Kobayashi chains joining $x$ and $y$ of length $d_n$ and such that $\lim_n d_n = d_\kob(x,y)$. 
Recall that each chain $C_n$ is given by analytic maps $f_1^{(n)}, \ldots, f_{N_n}^{(n)}: \D \to X$ and rigid points $z_1^{(n)}, \ldots, z_{N_n}^{(n)}$ in $\D$
satisfying the appropritate equalities.
For sufficiently large $n$ we may assume that 
\begin{equation*}
d_\kob (x,y) \le d_n < 2 \cdot  d_\kob (x,y) = r ~,
\end{equation*}
and in particular we see that $|z_l^{(n)}| <r $ for every $n\gg 0$ and every $1 \le l \le N_n$.
As a consequence of Lemma \ref{lema diam}, we see that $f_l^{(n)}(\bar{\D}(0; |z_l^{(n)}|  ) ) \subseteq \D^N(f_l^{(n)}(0); |z_l^{(n)}| C)$. Thus, for every $n\in \N$ we have 
$$d_{\mathbb{P}} (x, y) \le C \cdot d_n <   2 C \cdot d_\kob (x,y)  ~,$$
concluding the proof.
\end{proof}

\begin{obs}
Residue characteristic 0 is used  for the implications iii) $\Rightarrow$ i)  and iii) $\Rightarrow$ ii). 
\end{obs}

\begin{prop}
Let $X$ be a smooth projective variety defined over an algebraically closed non-Archimedean complete field $k$ of residue characteristic zero. 
If the Fubini-Study derivative of $\mathrm{Mor}_k(\D, X)$ is uniformly bounded in a neighbourhood of every rigid point, then
 the family  $\mathrm{Mor}_k(\D, X)$ is normal at every rigid point.
\end{prop}

\begin{proof}
Embed $X$ in some projective space $\mathbb{P}^{N, \an}$.
It suffices to prove the assertion for $z= 0$ in $\D$.
Assume first that  $U$ is a neighbourhood of $0$ on which  there exists a positive constant $C$ such that
\begin{equation*}
\sup_{\mathrm{Mor}_k(\D, X)} \sup_{U} |f^\prime (z)| \leq C < +\infty ~.
\end{equation*}
Pick any sequence of analytic maps $ f_n: \D \to X$.  Since $\mathbb{P}^{N, \an}$ can be covered by a finite number of closed polydisks isomorphic to $\bar{\D}^N$, we see that,  after maybe to extracting a subsequence and rescaling the image, the points $f_n(0)$ converge to a point in $\bar{\D}^N(0; \frac{1}{2})$, as $\mathbb{P}^{N, \an}$ is sequentially compact \cite{Poineau}.

 Now let $\eta_{0,r} \in U$ be the point corresponding to the closed all  $\bar{B}(0; r)$ in $k$.
  It follows from Lemma \ref{lema diam} that  $\diam f_n(\eta_{0,r}) \leq r \cdot C$ for all $n$. Choose $r >0$ such that $r \leq \frac{1}{2C}$  and set $U^\prime = \D(0; r)$.
   By continuity,  $f_n(U^\prime) \subseteq \bar{\D}^ N(0; \frac{1}{2})$, and by Theorem \ref{thm montel} there exists a subsequence converging  on $U^\prime$ to a continuous map. 
 \end{proof}

\subsection{Proof of Theorem \ref{THM EQUIVALENCES COURBES}}

The equivalence between i) and v) was proved in \cite{Cherrykobayashi}.
We provide a new proof of the fact that every smooth projective curve with positive genus is Cherry hyperbolic.

\smallskip

 iv) $\Rightarrow$ i):
 Pick any $| \lambda | >1$ and consider the sequence of analytic maps $f_n (z) = (\lambda z)^n$ from $\D$ to $\mathbb{P}^{1,\an}$. 
As explained at the beginning of this chapter, no subsequence of $f_n$ has a continuous limit,
and thus the family $\mor_k (\D, \mathbb{P}^{1,\an})$ is not normal.

\medskip

i) $\Rightarrow$ iv):
Assume  that $X$ has positive genus. 
Recall from \S \ref{section curves} that if $X$ is a smooth projective curve with positive genus, then $\San (X)$  is nonempty. 
The set $X \setminus \mathsf{S}^\an (X)$ is a disjoint union of infinitely many open disks.

Let $f_n: \D \to X$ be a sequence of analytic maps. By \cite[Theorem 4.5.3]{Berk}, the image of each map $f_n$ does not intersect $\San(X)$.

If the image of infinitely many maps $f_n$ is contained in the same connected component of $X \setminus \mathsf{S}^\an (X)$, then we may find a subsequence $f_{n_j}$ avoiding some  fixed connected component of the complement of $\San(X)$.
Hence, by Lemma \ref{lema aff domain} the maps $f_{n_j}$ 
take values in a fixed affinoid domain  of $X $ and so they converge pointwise to some continuous map by Theorem \ref{thm montel}.

Assume next that  at most finitely many $f_n(\D)$ are contained in the same connected component of $X \setminus \mathsf{S}^\an (X)$.
If  $\San(X)$ consists only of one point $\eta_X$, this means that the sequence $f_n$ converges pointwise to the constant map $\eta_X$.
Suppose otherwise $\San (X)$ is not a singleton.
Denoting by $r_X : X \to \mathsf{S}^\an (X)$ the usual retraction map, we consider the composition $y_n: = r_X \circ f_n$. Thus,  each map $y_n$ is  constant.
By compactness of $\San (X)$ we may find a subsequence $\{ y_{n_j} \}$ converging to some point $y \in \mathsf{S}^\an (X)$.

Fix an open neighbourhood $V \ni y$. By \cite[Th\'eor\`eme 4.5.4]{Ducroscourbes}, we are reduced to the following possibilities for $V$.
 If $y$ is a type III point, then $V$ is isomorphic to an open annulus whose skeleton is contained in $\San(X)$.
Otherwise, if $y$ has type II then $V\setminus \{ y \}$ is the disjoint  union of infinitely many open disks and  finitely many open annuli. In particular, the intersection of the skeleton of $V$ and $\San(X)$ is nonempty.
Pick any  $z \in \D$. 
For sufficiently large $n_j$, the points $f_{n_j} (z)$ lie in $V$. 
Thus, the subsequence $\{ f_{n_j} \}$ converges pointwise to the constant map $f \equiv y$.

\medskip

i) $\Rightarrow$ iii):
Let $X$ be a curve  of positive genus. We show that the Fubini-Study derivative of every map from $\D$ to $X$ is bounded.
Let  $f: \D \to X$ be an analytic map. 
By \cite[Theorem 4.5.3]{Berk}, the image of $f$ is contained in some connected component of the complement of $\San (X)$, i.e. in some open subset $V$ of $X$ analytically isomorphic to $\D$. 
 This implies that the Fubini-Study derivative of $f$ is bounded by $1$ on the whole disk.

\medskip

iii) $\Leftrightarrow$ ii):
This equivalence was shown in Theorem \ref{THM EQUIVALENCIAS}.

\medskip

ii) $\Rightarrow$ v):
This implication follows from Theorem \ref{THM EQUIVALENCIAS}.
\qed

\section{Curves with non-negative Euler characteristic}\label{section elliptic}

In this section we prove one of the implications of Theorem \ref{THM HYPERBOLICITY CURVES}.

\begin{prop}\label{prop elliptic}
Let $X$ be a curve with non-negative Euler characteristic.
Then there exists a one-dimensional basic tube $U$ such that the family $\mor_k(U, X)$ is not normal.
\end{prop}

Our proof follows \cite[Theorem 5.4]{FKT} and uses an equidistribution result for non-Archimedean elliptic curves from \cite{Petsche}.

\begin{proof}
Recall that a smooth algebraic curve satisfies $\chi(X) \ge 0$ if and only if it is isomorphic to one of the following models:
\begin{enumerate}
\item
$\mathbb{P}^{1, \an}$,  $\A^{1, \an}$,  $\A^{1, \an} \setminus \{ 0 \}$;
\item
 an elliptic curve.
\end{enumerate}
We claim that for any $\rho>1$ the family of analytic maps $$\mor _k \left(A\left(\frac{1}{\rho}, \rho\right), \A^{1, \an}\setminus \{ 0 \}\right)$$ is not normal.
To see this, consider the sequence  $f_n(z) = z^n$ from $ A(\frac{1}{\rho}, \rho)$ to $\A^{1, \an}\setminus \{ 0 \}$. 
Observe that $f_n(x_g) = x_g$ for all $n\in \N$, whereas $f_n(z) \to 0$ for any $|z|<1$. It follows that no limit map of $f_n$ can be continuous at the Gauss point. 
As a consequence of the definition of normality given at the begining of this chapter for families of maps whose target is a smooth algebraic curve, this proves the proposition for all cases in the first item. 

\medskip

Suppose now that $X$ is an elliptic curve.
Consider the map $f: X \to X$ induced by the multiplication by $2$.
Pick any point $x_0 \in \San (X)$, and suppose by contradiction that the family of the iterates $\{ f^n\}$
is normal on a neighborhood $U$ of $x_0$. 
Assume that the subsequence $f^{n_j}$ converges on $U$
to a continuous function $g: U \to X$. 

Choose any fixed rigid point $y \in X$ for $f$. 
By \cite[Theorem 1]{Petsche} the sequence of probability measures 
$4^{-n} (f^n)^\ast \delta_{y}$ converges to a probability measure $\mu$ whose support is equal to $\San(X)$, hence contains $x_0$.
We may thus find a sequence of rigid points $y_n \to x_0$ such that $f^n(y_n) = y$. 
Observe that $y_m \in U$ for sufficiently large $m$, thus $g(y_m) = \lim_{n_j} f^{n_j} (y_m) =y$, for all $m \in \N$.

But $f$ leaves the skeleton of $X$ invariant, hence $y=g(x_0) \in \San(X)$ which gives a contradiction.
\end{proof}

\section{Analytic maps on special domains}\label{section maps curves}

In this section, we study the normality of the family of analytic maps  taking values in a smooth projective curve $X$ having only one node that do not have good reduction, i.e. whose skeleton contains points different from the node.
Our discussion is based on the study of  three fundamental families of one-dimensional basic tubes: open disks, open annuli and star-shaped domains.

\subsection{Analytic maps  avoiding a type II point}

\begin{prop}\label{prop avoid nodes}
Let $X$ be a smooth irreducible projective curve and  $U$ any smooth  connected curve.
Let $\mathcal{F}$ be a family of analytic maps from $U$ to $X$.

If there exists a type II point $\eta \in X$ such that $\eta \notin f(U)$ for every $f \in \mathcal{F}$, then there exists an affinoid covering $( X_1, X_2 )$ of $X$  such that for every $f\in \mathcal{F}$, the image $f(U) $ is contained either in $X_1$ or in $X_2$.

Moreover, the affinoid cover $(X_i)$ is independent of $U$.
\end{prop}

\begin{cor}\label{cor disk annulus}
Let $X$ be a smooth irreducible projective curve of genus at least $2$ and $U$ be an open disk or an open annulus.
Then, there exists a finite affinoid cover $(X_i)$ of $X$ such that the image of every analytic map $f: U \to X$ is contained in some affinoid  $X_i$.

Moreover, the affinoid cover $(X_i)$ is independent of $U$.
\end{cor}

\begin{cor}\label{cor many nodes}
Let $X$  be  a smooth irreducible projective curve having at least two nodes and let $U$ be any smooth connected boundaryless curve.
Then, there exists a finite affinoid cover $(X_i)$ of $X$ and a locally finite open cover $(U_j)$ of $U$ by basic tubes such that for every analytic map $f: U \to X$ and every element of the cover $U_j$, the image $f(U_j) $ is contained in some affinoid $X_i$.

Moreover, the affinoid cover $(X_i)$ is independent of $U$.
\end{cor}

\begin{proof}[Proof of Proposition \ref{prop avoid nodes}]
Let $X$ be any smooth irreducible projective curve and $U$ any smooth connected curve. 
Let $\mathcal{F}$ be a family of analytic maps in $\mor_k (U , X)$  whose images avoid some type II point $\eta \in X$.
 
 The image of every $f \in\mathcal{F}$ is contained in some connected component of $X \setminus \{ \eta \}$.
We may thus pick any two distinct connected components $B_1, B_2$ of $X \setminus \{ \eta \}$. 
Then, $X_1 := X\setminus B_1$ and $X_2 = X\setminus B_2$ are  affinoid domains of $X$ by Lemma \ref{lema aff domain}, and 
 $(X_1, X_2)$ is a cover of $X$ safisfying the required property.
\end{proof}

\begin{proof}[Proof of Corollary \ref{cor disk annulus}]
Let $X$ be a smooth irreducible projective curve of genus at least two.
The curve $X$ contains at least one node $\eta$ by \eqref{eq genus}.
 On the other hand, if $U$ is an open disk or an open annulus, then it has no nodes.
By Lemma \ref{lema preimage node}, every analytic map $f: U \to X$ avoids $\eta$. The result follows from Proposition \ref{prop avoid nodes}.
\end{proof}

\begin{proof}[Proof of Corollary \ref{cor many nodes}]
Let $U$ be any smooth irreducible boundaryless curve. Recall that its set of nodes is discrete.
Consider a locally  finite open  cover $(U_j)_{j \in J}$ of $U$, where each $U_j$ is either an open disk, an open annulus or a star-shaped domain.
In particular, every basic tube $U_j$ contains at most one node.

Let $X$ be a smooth projective curve with at least two nodes, and denote $\mathsf{N}(X) = \{ \eta_1, \ldots, \eta_a\}$. 
For every $1\le l \le a$, let $B_l^1$ and $B_l^2$ be two distinct connected components of $X \setminus \{ \eta_l\}$ that are isomorphic to an open disk. The sets $X_l^1 = X \setminus B_l^1$ and $X_l^2 = X \setminus B_l^2$ are affinoid domains of $X$ by Lemma \ref{lema aff domain}.
The sets $(X_l^i)_{\substack{i=1,2\\1\le l \le a}}$ form a finite affinoid cover of $X$.

 Fix some $U_j$. For every  fixed $1\le l \le a$, consider the family of analytic maps $\mathcal{F}_{l,j} = \{ f: U \to X \mbox{ analytic } : \eta_l \notin f(U_j)  \}$. 
By Lemma \ref{lema preimage node}, we have that 
 $$\bigcup_{\substack{1\le l \le a \\ j \in J}} \mathcal{F}_{l,j} = \mor_k(U, X)~.$$
 We conclude by applying Proposition \ref{prop avoid nodes} to every family $\mathcal{F}_{l,j}$.
\end{proof}

\subsection{Analytic maps into curves having only one node}

\begin{prop}\label{prop one node}
Let $X$ be a smooth irreducible projective curve over $k$ of genus at least $2$  having a unique node $\eta_X$.
Assume further that the first Betti number of the skeleton of $X$ is at least $1$.
Let $U$ be a smooth connected boundaryless curve.

Then there exists a finite affinoid cover $(X_i)_{i\in I}$ of $X$ and a locally finite open cover $(U_j)_{j \in J}$ of $U$ by basic tubes such that for every analytic map $f: U \to X$ and every $j\in J$, the set $f(U_j)$ is contained in some affinoid domain $X_i$.

Moreover, the affinoid cover $(X_i)$ is independent of $U$.
\end{prop}

\begin{proof}
Let $b := b_1 (\San (X)) > 1$ be the first Betti number of the skeleton of $X$. 
As $b \ge 1$ and $X$ has only one node, the skeleton of $X$ consists of  $b$ loops $C_1, \ldots , C_b$ passing through  $\eta_X$, and so there are exactly $2b$ non-discal tangent directions at $\eta_X$.

The curve $X$ may be decomposed as a disjoint union of $\{ \eta_X \}$, open annuli $A_1, \ldots , A_b$ and infinitely many open disks by Lemma \ref{lema X minus N(X)}.
Fix some  $1 \le i \le b$.
 Pick an isomorphism $\varphi_i : A(R_i, 1) \to A_i$ with $R_i < 1$ and such that $\lim_{r \to 1} \varphi_i (\eta_{0,r}) = \lim_{r \to R_i} \varphi_i (\eta_{0,r}) = \eta_X$.
Consider the type II point $x_i := \varphi_i ( \eta_{0,\sqrt{R_i}})$, which lies on the loop $C_i$.
Pick any connected component $B_i$ of $X \setminus \{ x_i \}$ isomorphic to an open disk. The set $X_i := X \setminus B_i$ is an affinoid domain of $X$ by Lemma \ref{lema aff domain} and contains the point $\eta_X$.
The family $(X_i)_{1 \le i \le b}$ forms an affinoid cover of $X$.

\smallskip

Since $U$ is paracompact (cf. Theorem \ref{paracompact}), it suffices to show that for every point $z \in U$ there exists an open neighbourhood $V_0$ of $z$ such that for every analytic map $f: U \to X$, there exists some affinoid domain $X_i$ in the cover of $X$ such that $f(V_0) \subset X_i$.
Moreover, since the cases of disks and annuli have been treated separately in Corollary \ref{cor disk annulus}, obtaining an affinoid cover of $X$ similar to $(X_i)$, we may assume that $z\in \mathsf{N}(U)$.
We aim to construct a star-shaped domain $V_0 \subset U$ containing $z$ and such that every analytic map $f: V_0 \to X$ sends $V_0$ to some affinoid domain $X_i$.

Let $V$ be the connected component of the complement of $\mathsf{N}(U) \setminus \{ z \}$ in $U$ containing $z$.
It is a star-shaped domain in $U$ containing $z$  whose only node $\eta_V$ is precisely $z$.
Let   $\{ \vec{v}_1, \ldots , \vec{v}_a \}$ be the set  of non-discal directions at $\eta_V$. 
Fix some $1 \le j \le a$ and let  $I_j$ be the connected component of $\San (V) \setminus \{ \eta_V \}$ corresponding to the direction $\vec{v}_j$. It is isomorphic to an open segment.
The set $U(\vec{v}_j)$ is isomorphic to  an open annulus whose skeleton is precisely $I_j$.
We fix an isomorphism $\psi_j : A(\rho_j, 1) \to U(\vec{v}_j)$ with $\rho_j <1$ and such that $\psi_j$ extends continuously to the Gauss point $x_g$ satisfying $\psi_j ( x_g) = \eta_V$.

\medskip
Pick any analytic map $f: U \to X$. 
If $f(V)$ avoids the node $\eta_X$, then it is contained in some annulus $A_i$ or in some connected component of the complement of $\eta_X$ that is isomorphic to an open disk, hence in 
some affinoid  domain $X_i$.

Assume otherwise that $f $ maps the unique node $z=\eta_V \in V$ to $\eta_X$.

Consider the  star-shaped domain $V_0$ contained in $V$
obtained by reducing every  segment $I_j$ in such a way that the resulting segment $I_j^0$ is isomorphic to $(\rho^0_j, 1)$ with
$\rho^0_j \ge \max_{1\le j \le a} \sqrt{\rho_j}$.
We claim that $x_i$ does not belong to $f (I_j^0)$  and that $f(V_0) \subseteq X_i$ for some index $i$.

Recall that the  tangent map $df (\eta_V): T_{\eta_V} V \to T_{\eta_X} X$  is a rational map  on the residue curve at $\eta_V$ with values in the residue curve at $\eta_X$, which is surjective.
The preimage of every  non-discal direction at $\eta_X$ consists only of non-discal directions at $\eta_V$ by
Lemma \ref{lema discal}.
 We may thus choose $j$ such that  $df (\eta_V) (\vec{v}_j) \in T_{\eta_X} X$ is non-discal.

The restriction of $f$ to $U(\vec{v}_j)$ takes values in some annulus $A_i \subset X$, and so we may consider the composition $F_{i,j} = \varphi_i^{-1} \circ f \circ \psi_j : A(\rho_j, 1) \to A(R_i, 1)$.
 Since $f(\eta_V) =\eta_X$, we see that  $\lim_{r \to 1} F_{i,j}(\eta_{0,r}) = x_g$.
Additionally, $F_{i,j} (\San (A(\rho_j, 1) ) \subseteq \San ( A(R_i, 1))$ by Proposition \ref{cor if you touch the skeleton you stay there}.

The map $F_{i,j} $ can be expanded into a Laurent series $F_{i,j} (z) = \sum_{n \in \Z} a_n z^n$. 
Consider the real function $ \theta_{i,j} (r) :=\max_{n \in \Z} \{ \log |a_n| + n r \}$, defined on the open real interval $( \log \rho_j, 0)$.
Since $F_{i,j}$ is an analytic function on an open annulus without zeroes, there exists an integer $n_0 \in \Z$ such that the function $\theta_{i,j}$ is of the form $\theta_{i,j}(r)= \log  | a_{n_0}|  + n_0 r$.
As $F_{i,j}$ extends continuously to the Gauss point in $A(\rho_j, 1)$ with $\lim_{r \to 1} F_{i,j}(\eta_{0,r}) = x_g$, we see that $| a_{n_0}| = 1$. It follows that $\theta_{i,j}$ extends continuously to the origin $0  \in \R$  with $\theta_{i,j} (0) = 0$.

 Observe that $\theta_{i,j} (r) \ge \log R_i$ for every $r \in ( \log \rho_j, 0)$ by the definition of $F_{i,j}$.
It follows that the graph of $\theta_{i,j}$ lies above the linear function  $r \in ( \log \rho_j, 0) \mapsto \frac{ \log R_i}{ \log \rho_j} r$.
In particular, we see that    $  n_0 \le \frac{ \log R_i}{ \log \rho_j} $.
We conclude that $\theta_{i,j} (r) > \frac{1}{2} \log R_i$ as soon as $r> \log \sqrt{\rho_j}$. Notice that this condition does not depend on $R_i$.

Pick any $1> \rho^0_j > \max_{1\le j \le a} \sqrt{\rho_j}$ and  reduce the segment $I_j$ into a segment $I_j^0$ such that the corresponding open subset $U(\vec{v}_j^0) \subset V$ is isomorphic to  $\psi_j (A(\rho_j^0, 1))$.
The previous calculations show that  the image under $f$ of the segment $I_j^0$ covers at most half the loop $C_i$ at $\eta_X$ starting with the direction $df (\eta_U)(\vec{v}_j)$, avoiding the point $x_i \in X$.
It follows that $f( U(\vec{v}_j^0)  )\subseteq X_i$.

We may carry over  this procedure to every non-discal direction at $\eta_V$, imposing that  $\rho^0_j > \max_{1\le j \le a} \sqrt{\rho_j}$ for every $j=1, \ldots, a$. Let $V_0$ be the resulting star-shaped domain in $U$, which contains the point $z$.
 We conclude that the restriction of any $f: U \to X$ to $V_0$ takes values in some affinoid domain $X_i$.
\end{proof}

\section{Proof of Theorem \ref{THM NORMAL BAD REDUCTION}}\label{section thm d}

From now on, we will suppose that the base field $k$ has zero residue characteristic.

\subsection{De Franchis theorem}

At several stages of the proofs of Theorems \ref{THM NORMAL BAD REDUCTION} and \ref{THM HYPERBOLICITY CURVES}, we shall need the following (slight) improvement of the original De Franchis theorem that applies to nonproper curves.
We refer for example to \cite{Tsushima} for a purely algebraic proof in a much more general context in arbitrary dimension. 

\begin{theorem}\label{de franchis}
Let $\tilde{k}$ be any algebraically closed field of characteristic zero. Let $X$ and $Y$ be two smooth algebraic curves defined over $\tilde{k}$. Suppose that $\chi(X) <0$, with $X$ not necessarily proper.
Then the set of regular maps from $Y$ to $X$ is finite. 
\end{theorem}

\subsection{The compact case}

Let $X$ be a smooth projective curve of genus at least $2$ not having good reduction. Recall that the latter condition means that its skeleton consists of more than one point. 
Since $g(X) \ge 2$, then we know that $\mathsf{N}(X)$ is nonempty.

Let   $U$ be  any smooth boundaryless curve.
If $\mathsf{N}(X)$ consists of more than one point, then we conclude by Corollary \ref{cor many nodes}. Otherwise, we are exactly in the situation of Proposition \ref{prop one node}.
\qed

\subsection{General algebraic case}

Let $X$ be a smooth irreducible algebraic curve with negative Euler characteristic whose skeleton is not a singleton, and let $\bar{X}$ be the unique smooth projective curve such that there exists an open embedding $X\to \bar{X}$ with $\bar{X}\setminus X$ a finite set of rigid points.
Let $U$ be any smooth connected boundaryless curve.

Our aim is to construct a finite affinoid cover $(\bar{X}_i)$ of $\bar{X}$ and a locally finite cover $(U_j)$ of $U$ such that for every analytic map $f: U \to X$ there exists some $\bar{X}_i$ with $f(U_j) \subseteq \bar{X}_i$.

\smallskip

Recall that the  non-proper algebraic  curves $X$ with negative Euler characteristic are  $\mathbb{P}^{1,\an}$ with at least three rigid points removed, and  elliptic curves  and  curves with genus at least  $2$ with finitely many rigid points removed. 

The case of $\mathbb{P}^{1, \an} \setminus \{0, 1, \infty\}$ is   treated in \cite[Proposition 3.2]{FKT}. 
If $X$ is such that $\bar{X}$ does not have good reduction and 
its genus is greater than $1$, then we are reduced to the projective case, which has already been treated.

Therefore, it only remains to address the cases where  $X $ is either an elliptic curve  or a  projective curve with good reduction $\bar{X}$ with a rigid point removed.
 The curve $X$ has exactly one node $\eta_X$, which is a branching point of the skeleton in the case where $\bar{X}$ has bad reduction 
 and  a point of positive genus if
  $\bar{X}$ has  good reduction.
   We shall make no distinction in the genus of $\bar{X}$ when dealing with the good reduction case.
 
\smallskip

  If $\bar{X}$ is an elliptic curve with bad reduction, then the skeleton of $X$
  consists of a loop $C$ passing through $\eta_X$ and the segment joining $\eta_X$ and the unique point in $\bar{X}\setminus X$ with its endpoint removed.
 We are in a situation similar to that  of Proposition \ref{prop one node}.
We obtain the following  affinoid cover $(\bar{X})_{i=1, 2}$ of $X$.
Let $x \in \San(X)$ be any type II point on the loop $C \subset \San(X)$ different from $\eta_X$ and pick any connected component $B_1$ of $X \setminus \{ x \}$ that is isomorphic to an open disk. We set $\bar{X}_1 : = \bar{X} \setminus B_1$.
  Set $\bar{X}_2 : = \bar{X} \setminus B_2$, where $B_2$ is a connected component of $X \setminus \{ \eta_X \}$ isomorphic to a disk.
We obtain a finite open cover of $U$ by basic tubes   $U_j$ having at most one node.
  
  \smallskip
 
 Assume now that $\bar{X}$ is a curve with good reduction and genus at least $1$. 
There is only one non-discal direction $\vec{w}$ at $\eta_X$, and the set $U (\vec{w})$ is isomorphic to a punctured disk.
Pick any analytic map $f: U \to X$.

 Consider the following affinoid cover $(\bar{X}_1, \bar{X}_2)$ of $\bar{X}$.
 Pick an open subset $B_1$ of  $\bar{X}$ isomorphic to an open disk and containing the unique point in $\bar{X} \setminus X$. Assume further that  $B_1 \cap X$ is strictly contained in   $U (\vec{w})$ and
 set $\bar{X}_1 := \bar{X} \setminus B_1$.
Let $B_2$ be a connected component of $X\setminus \{ \eta_X \}$ isomorphic to a disk and set $\bar{X}_2 := \bar{X} \setminus B_2$.

If $f$ avoids the point $\eta_X$, then $f(U)$ is clearly contained either in $\bar{X}_1$ or in $\bar{X}_2$.
We may thus assume that there exists some $\eta_j \in \mathsf{N}(U)$  such that $f(\eta_j)  = \eta_X$.

Let $U_j \subset U$ be the connected component of the complement of $\mathsf{N}(U) \setminus \{ \eta_j \}$ in $U$ containing $\eta_j$.
The tangent map $df(\eta_j)$ is surjective, and by Lemma \ref{lema discal} every preimage of $\vec{w}$ is non-discal.
We may thus pick  $\vec{v} \in T_{\eta_j} U$ non-discal  such that $df(\eta_j) (\vec{v}) = \vec{w}$.
Fix isomorphisms $\varphi: \D \setminus \{ 0 \} \to U(\vec{w})$ and $\psi : A (\rho , 1) \to U (\vec{v})$ with $\rho <1$. Assume further  that both extend continuously to the Gauss point, with $\lim_{r \to 1} \varphi (\eta_{0,r} ) = \eta_X$ and $\lim_{r \to 1} \psi (\eta_{0,r} ) = \eta_X$.

The composition  $F = \varphi^{-1} \circ f \circ \psi $ 
 is so an analytic map on an open annulus with values in the punctured disk.
 Write $F(z) = \sum_{n \in \Z} a_n z^n$ and consider the real function $\theta (r) = \max_{n \in \Z} \{ \log |a_n| +  nr \}$ on the open interval $(\log \rho, 0)$. 
 Since $F$ has no zeros, there exists an integer $n_0$ such that $\theta (r ) = \log |a_{n_0}| +  n_0 r$ for all $r \in (\log \rho, 0)$.
 Moreover, $|a_{n_0}| = 1$, as $F(x_g) = x_g$.
 
 Consider the  tangent map $df(\eta_j) : T_{\eta_j} U_j \to T_{\eta_X} X$.
 Both $T_{\eta_j} U_j $ and $T_{\eta_X} X$ are isomorphic to smooth projective curves $C_{U_j}$ and $C_X$ respectively over $\tilde{k}$ with a finite number of marked points, corresponding to the non-discal directions at $\eta_j$ and $\eta_X$ respectively.
  In particular, $C_X$ is a curve with one marked point and genus at least $1$.
 The inverse image  under $df(\eta_j)$ of the marked point in $C_X$ is contained in the set of marked points in $C_{U_j}$ by Lemma \ref{lema discal}.
Applying Theorem \ref{de franchis} to the curve $C_X$ with the marked point, we obtain that there are only finitely many possibilities for the tangent map $df(\eta_j)$.
As a consequence, the degree at every  marked point in $C_{U_j}$ of the rational map $df(\eta_j)$ is bounded. 
At the marked point  corresponding to $\vec{v} \in T_{\eta_j} U_j$, this degree is precisely $n_0$, the slope of  $\theta$, which is thus bounded.

Hence, after maybe reducing  the basic tube $U_j$ we see that $f(U_j) \subset \bar{X}_1$.
Repeating this procedure at every node of $U$, we obtain a locally finite open cover $(U_j)$ of $U$ consisting of open disks, open annuli and star-shaped domains satisfying the required property.
\qed

\begin{obs}
The previous  arguments in the case where $\bar{X}$ has good reduction apply verbatim to $\mathbb{P}^{1, \an} \setminus \{0, 1, \infty\}$, since its skeleton is a tripod joining the points $0,1$ and $\infty$ and  $\mathsf{N}(\mathbb{P}^{1, \an} \setminus \{0, 1, \infty\} ) = \{x_g \}$.
\end{obs}

\section{Proof of Theorem \ref{THM HYPERBOLICITY CURVES}}\label{section thm c}

Recall that the base field $k$ is assumed to have zero residue characteristic. 

\subsection{Curves having good reduction}

Recall that the skeleton of a smooth projective curve $X$ with good reduction consists of a single point $\eta_X$ whose genus equals that of $X$.

Our previous arguments do not apply in the case of smooth projectve curves having good reduction, 
  and we therefore treat this case separately.

\begin{prop}\label{prop good reduction}
Let $k$ be an algebraically closed complete field of zero residue characteristic.
Assume that the residue field is countable. 
Let $X$ be a smooth irreducible projective curve over $k$ with good reduction and of genus at least $2$, and let  $U$ be a star-shaped domain.
Then the family $\mor_k(U, X)$  is normal. 
\end{prop}

We shall need for the proof the following weaker version of \cite[Lemma 3.6.8]{Differentfunction}

\begin{lema}\label{lema monomial}
Let $k$ be an algebraically closed field of zero residue characteristic.
Let $U$ be the complement of finitely many closed disks in $D$.
Let $f: U \to \D$ be an analytic map.
Let $I \subset U$ be an interval.

Then there exists a finite subdivision of $I$ into smaller intervals $I_j$  such that 
$\diam (f (z)) = a_j \diam(z)^{n_j}$ for every $z \in I_j$.
\end{lema}

\begin{proof}[Proof of Proposition \ref{prop good reduction}]
Let $f_n: U \to X$ be a sequence of analytic maps. 
We reduce to the case where the only node  $\eta_U$ in $U$ is mapped by every $f_n$ to the only node $\eta_X$ in $X$ by Proposition \ref{prop avoid nodes}.

The formula \eqref{eq genus} assures that $g(\eta_X) = g(X) \ge 2$.
For every $n\in \N$, the tangent map $df_n (\eta_U): T_{\eta_U} U \to T_{\eta_X}X$ is a rational map between the residue curves at $\eta_U$ and $\eta_X$. 
By Theorem \ref{de franchis} there are only finitely many such nonconstant maps, as the residue curve at $\eta_X$ has genus greater than $1$.
We may thus assume that all the tangent maps $df_n (\eta_U)$ are equal. 
 Let $d$ be the degree of the $df_n (\eta_U)$.
 
We  treat every connected component of $U \setminus \{ \eta_U \}$ separately.
Pick any tangent direction $\vec{v}$ at $\eta_U$. The image of $U( \vec{v})$ under every $f_n$ is contained in some fixed connected component $V$ of $X \setminus \{ \eta_X \}$, as all the tangent maps $df(\eta_U)$ agree. 
Thus, $f_n (U( \vec{v}))$ is contained in some affinoid domain of $X$ for every $n$. Theorem \ref{thm montel} implies that there exists a subsequence $f_{n_j}$ converging on $U(\vec{v})$ to some continuous map $f_\infty$.
Since $\tilde{k}$ is countable, we may extract diagonally at every tangent direction at $\eta_U$ and obtain a limit map $f_\infty : U \to X$ that is continuous on $U \setminus \{ \eta_U \}$.

\smallskip

Observe that $f_\infty (\eta_U) = \eta_X$.
It remains to check that $f_\infty$ is continuous at $\eta_U$. In order to do so, it suffices to  verify that for every sequence of points $z_m \in U$ converging to $\eta_U$ we have $f_\infty (z_m) \stackrel{m \to \infty}{\rightarrow}\eta_X$.
If the points $z_m$ belong to infinitely many different connected components of $U\setminus \{ \eta_U \}$, then their images $f_\infty (z_m)$ belong to infinitely many distinct connected components of $X \setminus \{ \eta_X \}$  and we conclude.

We may thus assume that all the points $z_m$ belong  to $U(\vec{v})$ for some fixed tangent direction  $\vec{v} \in T_{\eta_U} U$.
Fix an isomorphism $\psi: Y \to U(\vec{v})$, where $Y$ is an open disk or an open annulus depending on whether $\vec{v}$ is discal or not. Assume that $\psi$ extends continuously to  the Gauss point $x_g$ with $\lim_{r \to 1} \varphi (\eta_{0,r} ) = \eta_U$.
We may assume that for every $m \in \N$ there exists some $0 \le r < 1$ such that  $z_m = \psi (\eta_{0,r})$.
Fix an isomorphism $\varphi: \D \to V$ that extends continuously to $x_g$ with  $\lim_{r \to 1} \varphi (\eta_{0,r} )  =  \eta_X$.
Checking the continuity of $f_\infty$ at $\eta_U$ amounts to showing that $\lim_{r \to 1} \varphi^{-1} \circ f_\infty \circ \psi (\eta_{0,r} ) = x_g$.

\smallskip

For every fixed $n \in \N \cup \{ \infty \}$, set $F_n = \varphi^{-1} \circ f_n \circ \psi : Y \to \D$.  
The series development $F_n (z) = a_0^{(n)} + \sum_{i \neq 0} a_i^{(n)} z^i$ is such that $|a_i^{(n)}| \le 1$ for all $i \in \Z$, $|a_d^{(n)} | = 1$ and $|a_i^{(n)} | <1$ for $i \le d$, 
where $d$ denotes the degree of $df_n(\eta_U)$.
Consider the segments  $I = \{ \eta_{0,r} : 0 \le r < 1\} \subset \D$
and  $l_n = I \cap F_n^{-1}(I) $. 
Notice that $\diam (F_n( \eta_{0, r_n}) ) = |a_0^{(n)} | $, where $r_n = \inf \{ r : \eta_{0,r} \in l_n \}$.

We distinguish two cases. 
Assume first that there exists a positive real number $R <1$ such that $|a_0^{(n)} | \le R$ for  infinitely many indices $n \in \N$.
 After maybe extracting a subsequence, we may reduce $Y$ as to obtain an annulus $Y_1$ centered at $0$  containing the point  $x_g$ in its topological boundary and such that every map
$F_n$ avoids  the disk $\bar{\D}(0; R)$. 
Moreover, the skeleton $\San(Y_1) \subseteq \bigcap_n l_n$ is mapped to the segment $I$.
For sufficiently large $r$, Lema \ref{lema monomial} implies that $F_n ( \eta_{0,r} ) = \eta_{0,rd}$, and so we see that $\lim_{r \to 1} F_\infty (\eta_{0,r}) = x_g $.

 \smallskip

Suppose next that $|a_0^{(n)}| \to 1$. 
In this case, there exists an open annulus $A \subseteq U(\vec{v})$ whose topological boundary contains $\eta_U$ such that  the restriction of $f_\infty$ to $A$ is the constant map $\eta_X$.
Indeed, after extracting a subsequence we have that $F_n( \eta_{0,r_n}) = \eta_{0, r_n d} = \eta_{0, |a_0^{(n)}| }$ by Lemma \ref{lema monomial}.
 As $ |a_0{(n)}| \to 1$, then $r_n \to \frac{1}{d}$.
Pick any $r > \frac{1}{d}$. For $n \gg 0$, we may assume that $r > r_n$. Then, $F_n( \eta_{0,r}) =   \eta_{0, rd}$, and since $rd > r_n d = |a_0^{(n)}| \to 1$, we see that $F_n (\eta_{0,r})  \stackrel{n\to \infty}{\rightarrow}x_g$.
Thus, $F_\infty\equiv x_g$ on the open annulus $A( \frac{1}{d}, 1)$.
\end{proof}

\subsection{Proof of Theorem \ref{THM HYPERBOLICITY CURVES}}
Assume first that $X$ is a  projective curve of   genus at least $2$ with good reduction.
Denote by $\eta_X$ its only node.
Recall that $X\setminus \{ \eta_X \}$ is a countable disjoint union of open disks.
Since the open disk and the open annulus have been treated separatedly in Corollary \ref{cor disk annulus}, we may assume that $U$ has nodes. 
The set $\mathsf{N}(U)$ is  discrete and consequently we may find a locally finite  cover $(U_j)$ of $U$ by basic tubes   being either an open disk, an open annulus or a star-shaped domain.

Fix some basic tube $U_j$. If the set of nodes of $U_j$ is empty, then $U_j$ is either a disk or an annulus, and we may apply   Corollary \ref{cor disk annulus} and Theorem \ref{thm montel} to extract a subsequence converging pointwise on $U_j$ to some continuous map.
Suppose now that $U_j$ contains one node $\eta_j$, i.e. it is a star-shaped domain. Then we  apply Proposition 
 \ref{prop good reduction} to extract a subsequence  that is pointwise converging on $U_j$ to some continuous map. 
This procedure may be repeated for every open set $U_j$, and extracting diagonally we obtain a subsequence $f_{n_j}$ that converges on $U$ to some continuous map $f_\infty : U \to X$.
This concludes the proof in the case of curves having good reduction.

\medskip

Suppose now that $X$ is a smooth algebraic curve with negative Euler characteristic whose skeleton is not a single point.
Let $\bar{X}$ be the smooth projective curve such that 
$X$ can be embedded in $\bar{X}$ and $\bar{X} \setminus X$ is a finite set of rigid points.

Let $U$ be any smooth connected boundaryless curve and pick any sequence of analytic maps $f_n : U \to X$. By Theorem \ref{THM NORMAL BAD REDUCTION} we may find a locally finite open cover $U_j$ of $U$ by basic tubes, and a finite $k$-affinoid cover $\bar{X}_i$ of $\bar{X}$ such that for every $n\in \N$ and every $j$ one has $f_n(U_j) \subseteq \bar{X}_i$ for some $i$.
By Theorem \ref{thm montel} and a diagonal extraction argument, we may extract a subsequence converging pointwise to some continuous map $f_\infty : U \to \bar{X}$. 

By Proposition \ref{new isolated zeros}, for each
index $j$ either we have  $f_\infty (U_j) \subset X$ or $f_\infty |_{U_j}$ is constant equal to some point in  $\bar{X} \setminus X$. 
If $f_\infty (U) $ is not included in $X$, then by continuity and connectedness we conclude that $f_\infty$ is constant equal to some
point in  $\bar{X} \setminus X$ as required.
\qed

\bibliographystyle{alpha}

\bibliography{biblio}

\end{document}